\documentclass[a4paper,10pt,reqno]{amsart}
\usepackage{amssymb,latexsym}
\usepackage{mltex}
\usepackage{amsmath}
 \usepackage{lscape}
\usepackage{enumerate}
\usepackage{amsthm}
\usepackage{hyperref}
\newtheorem{thm}{Theorem}
\newtheorem{cor}{Corollary}
\newtheorem{lem}{Lemma}[section]
\newtheorem{prop}[lem]{Proposition}
\newtheorem{rem}[lem]{Remark}

\newcommand{\nablab}{\overline{\nabla}}

\newcommand{\lgra}{\longrightarrow}

\newcommand{\iid}{\mathrm{Id}\,}

\newcommand{\trace}{\mathrm{tr\,}}

\newcommand{\lto}{\ensuremath{\longrightarrow}}
\newcommand{\C}{\mathbb{C}}

\newcommand{\HH}{\mathbb{H}}
\newcommand{\HC}{\mathbb{H}^{\mathbb{C}}}

\newcommand{\R}{\mathbb{R}}

\newcommand{\Z}{\mathbb{Z}}

\newcommand{\pre}{\Re e}

\newcommand{\tch}{\hat}

\newcommand{\bigtch}{\widehat}

\newcommand{\function}[5]
{\begin{eqnarray*}\begin{array}{r@{}ccl}
 #1\;\colon\;  & #2 &\lto & #3 \\[.05cm]
  & #4 &\longmapsto  & #5
\end{array}\end{eqnarray*}
}
\newcommand{\beqt}{\begin{equation}}  \newcommand{\eeqt}{\end{equation}}
\newcommand{\bal}{\begin{align}}      \newcommand{\eal}{\end{align}}
\newcommand{\ba}{\begin{array}}      \newcommand{\ea}{\end{array}}
\newcommand{\bc}{\begin{center}}     \newcommand{\ec}{\end{center}}
\newcommand{\be}{\begin{enumerate}}  \newcommand{\ee}{\end{enumerate}}
\newcommand{\beq}{\begin{eqnarray}}  \newcommand{\eeq}{\end{eqnarray}}
\newcommand{\beQ}{\begin{eqnarray*}} \newcommand{\eeQ}{\end{eqnarray*}}
\newcommand{\bi}{\begin{itemize}}    \newcommand{\ei}{\end{itemize}}
\newcommand{\bt}{\begin{tabular}}    \newcommand{\et}{\end{tabular}}
\newcommand{\finpreuve}{\hfill\square\\}

\newcommand{\1}{\textit{1}}
\begin{document}

\title{On the spinorial representation of spacelike surfaces into 4-dimensional Minkowski space}

\author{Pierre Bayard}

\address{Pierre Bayard, Instituto de F\'isica y Matem\'aticas. U.M.S.N.H. Ciudad Universitaria. CP. 58040 Morelia, Michoac\'an, M\'exico}
  
\email{bayard@ifm.umich.mx}


\date{}


\begin{abstract}
We prove that an isometric immersion of a simply connected Riemannian surface $M$ in four-dimensional Minkowski space, with given normal bundle $E$ and given mean curvature vector $\vec H\in\Gamma(E),$  is equivalent to a normalized spinor field $\varphi\in\Gamma(\Sigma E\otimes\Sigma M)$ solution of a Dirac equation $D\varphi=\vec H\cdot\varphi$ on the surface. Using the immersion of the Minkowski space into the complex quaternions, we also obtain a representation of the immersion in terms of the spinor field. We then use these results to describe the flat spacelike surfaces with flat normal bundle and regular Gauss map in four-dimensional Minkowski space, and also the flat surfaces in three-dimensional hyperbolic space, giving spinorial proofs of results by J.A. G\'alvez, A. Mart\'{\i}nez and F. Mil\'an.
\end{abstract}

\maketitle


\markboth{PIERRE BAYARD}{}
\section*{Introduction}
A conformal immersion of a surface in $\R^3,$ with mean curvature $H,$ may be represented by a spinor field $\varphi$ solution of the Dirac equation $D\varphi=H\varphi$ on the surface; this fact generalizes the classical Weierstrass representation of minimal surfaces, and has been studied by many authors; we refer to \cite{KS,Ta1,Fr}, and the references therein. The representation of surfaces using spinors has been subsequently extended to other geometric contexts, especially to the Riemannian and semi-Riemannian space forms, in dimension three \cite{Mo,La,Ro,LR}, and in dimension four \cite{Ko,KL, Va,Ta2,BLR}; close to our subject, the paper \cite{Va} is concerned with the spinorial representation of conformal immersions of surfaces into 4-dimensional pseudo-Riemannian manifolds, and in particular into 4-dimensional Minkowski space.

In the present paper, we are interested in the spinorial description of the isometric immersions of a Riemannian surface into 4-dimensional Minkowski space $\R^{1,3},$ with given normal bundle and given mean curvature vector (this is the approach followed in \cite{BLR}), and in its applications to the geometry of spacelike surfaces in $\R^{1,3}$, especially to the description of the flat surfaces: using such a representation we will obtain spinorial proofs of results by J.A. G\'alvez, A. Mart\'{\i}nez and F. Mil\'an concerning the representation of the flat surfaces with flat normal bundle in $\R^{1,3}$ \cite{GMM} and also in three-dimensional hyperbolic space $\HH^3$ \cite{GMM1,GMM}. 
 
Let $(M,g)$ be a Riemannian surface. We suppose that $E$ is a vector bundle of rank $2$ on $M,$ equipped with a Lorentzian metric, i.e. a metric of signature $(1,1)$ on each fiber,  a compatible connection, and an orientation (in space, and in time); we assume moreover that spin structures are given on $E$ and on $TM,$ and we define $\Sigma:=\Sigma E\otimes \Sigma M,$ the tensor product of the corresponding bundles of spinors. Let $\HC$ be the space of complex quaternions, defined by
$$\HC:=\{q_0\1+q_1I+q_2J+q_3K,\ q_0,q_1,q_2,q_3\in\C\}$$
where $I,J$ and $K$ are such that
$$I^2=J^2=K^2=-\1,\ IJ=-JI=K.$$
We will see Section \ref{preliminaries} that two natural bilinear maps
$$H:\Sigma\times\Sigma\rightarrow\C\hspace{1cm} \mbox{and}\hspace{1cm}\langle\langle.,.\rangle\rangle:\Sigma\times\Sigma\rightarrow\HC$$
are defined on $\Sigma.$ The main result of the paper is the following:
\begin{thm}\label{th main result}
Suppose that $M$ is moreover simply connected, and let $\vec{H}$ be a section of $E$. The following statements are equivalent.
\begin{enumerate}
\item There exists a spinor field $\varphi$ of $\Gamma(\Sigma)$ with $H(\varphi,\varphi)=1$ solution of the Dirac equation
$$D\varphi=\vec{H}\cdot\varphi.$$
\item There exists a spinor field $\varphi\in\Gamma(\Sigma)$ with $H(\varphi,\varphi)=1$ solution of 
$$\nabla_X\varphi=-\frac{1}{2}\sum_{j=2,3}e_j\cdot B(X,e_j) \cdot\varphi, $$
where $B:TM\times TM\lgra E$ is bilinear with $\frac{1}{2}\trace(B)=\vec{H},$ and where $(e_2,e_3)$ is an orthonormal basis of $TM$ at every point.
\item There exists a (spacelike) isometric immersion $F$ of $(M,g)$ into $\R^{1,3}$ with normal bundle $E$ and mean curvature vector $\vec{H}$.
\end{enumerate}
Moreover, $F=\int\xi,$ where $\xi$ is the closed 1-form on $M$ with values in $\R^{1,3}$ defined by
$$\xi(X):=\langle\langle X\cdot\varphi,\varphi\rangle\rangle\hspace{.5cm}\in\hspace{.5cm}\R^{1,3}\subset\HC$$
for all $X\in TM.$ 
\end{thm}
We refer to Section \ref{preliminaries} for the definitions of the Clifford product $"\cdot"$ and of the Dirac operator $D$ acting on $\Gamma(\Sigma),$ and for the immersion of the Minkowski space $\R^{1,3}$ into $\HC.$

We then give various applications of this result. We first give a spinorial proof of the fundamental theorem:
\begin{cor}\label{corollary1 theorem}
If, additionally to the hypotheses in the theorem above, a bilinear map $B:TM\times TM\rightarrow E$ satisfying the Gauss, Codazzi and Ricci equations is given, then there is an isometric immersion of the surface in $\R^{1,3}$ with normal bundle $E$ and second fundamental form $B;$ the immersion is moreover unique up to the action of the group of Lorentzian motions in $\R^{1,3}.$
\end{cor}
As a second application, we explain how Theorem \ref{th main result} permits to recover the spinorial characterizations of the isometric immersions in $\R^3,$ in $\HH^3$ \cite{Mo} and in $\R^{1,2}$ \cite{LR}; for immersions in $\R^{1,2}$ we obtain a characterization which is different to the characterization given in \cite{LR}, and we also give a representation formula (see Remark \ref{rmq R1,2} below).
 
The third application of the theorem above is the description of the flat surfaces with flat normal bundle and regular Gauss map in four-dimensional Minkowski space: these surfaces are generally described by two holomorphic functions and two smooth functions satisfying a condition of compatibility;  this is the main result of \cite{GMM}, that we prove here using spinors. We set here this result as follows:
\begin{cor}\label{corollary2 theorem}
Let $(U,z)$ be a simply connected domain in $\C,$ and consider
$$\theta_1=f_1dz\hspace{1cm}\mbox{and}\hspace{1cm}\theta_2=f_2dz$$
where $f_1,f_2:U\rightarrow\C$ are two holomorphic functions such that $f_1^2+f_2^2\neq 0.$ Suppose that $h_0,h_1:U\rightarrow\R$  are smooth functions such that
\begin{equation}\label{alpha function f h cor}
\alpha_1:=\frac{1}{f_1^2+f_2^2}\left(ih_0f_1+h_1f_2\right)\hspace{.5cm}\mbox{ and }\hspace{.5cm}\alpha_2:=\frac{1}{f_1^2+f_2^2}\left(ih_0f_2-h_1f_1\right),
\end{equation}
considered as real vector fields on $U,$ are independent at every point and satisfy $\left[\alpha_1,\alpha_2\right]=0.$  Then, if $g:U\rightarrow Spin(1,3)\subset\HC$ is a map solving 
\begin{equation}\label{eqn g1}
dg\ g^{-1}=\theta_1 J+\theta_2 K,
\end{equation}
and if we set
\begin{equation}\label{th def xi}
\xi:=g^{-1}(\omega_1J+\omega_2K)\hat{g}
\end{equation}
where $\omega_1,\omega_2:TU\rightarrow\R$ are the dual forms of $\alpha_1,\alpha_2\in\Gamma(TU)$ and where $\hat{g}$ stands for the map $g$ composed by the complex conjugation $\hat{}$ of the four coefficients in $\HC,$ the function $F=\int\xi$ defines an immersion $U\rightarrow\R^{1,3}$ with $K=K_N=0.$ Reciprocally, the immersions of $M$ simply connected such that $K=K_N=0,$ with regular Gauss map and whose osculating spaces are everywhere not degenerate, are of this form.
\end{cor}

We refer to Section \ref{preliminaries} for notation and to Section \ref{ss recovering} for the proof. To obtain this result, we first study the Gauss map of a spacelike surface in four-dimensional Minkowski space, and we obtain the following results, which might be of independent interest: we first show that the Grassmannian of the spacelike oriented planes in $\R^{1,3}$ naturally identifies to a complex two-sphere in $\C^3,$ and we compute the expression of the pull-back by the Gauss map of the natural complex area form defined on it. This expression leads to simple extrinsic proofs of the Gauss-Bonnet and the Whitney's Theorems (Remark \ref{gauss-bonnet whitney}), and also readily furnishes the natural complex structure induced by the Gauss map on a flat spacelike surface with flat normal bundle and regular Gauss map; this complex structure coincides with the complex structure introduced in \cite{GMM}.

Finally, we describe the flat surfaces in three-dimensional hyperbolic space $\mathbb{H}^3,$ giving a spinorial proof of another result by J.A. G\'alvez, A. Mart\'{\i}nez and F. Mil\'an \cite{GMM1,GMM}.
\\

The problem of the global description of the flat surfaces with flat normal bundle and regular Gauss map in four-dimensional Euclidean space  is still an open question; see \cite{Gal} for some recent results. A spinorial approach as in Theorem \ref{th main result} and Corollary \ref{corollary2 theorem} above might give  some new insights to this subject. 
\\

The outline of the paper is as follows: Section \ref{preliminaries} is devoted to preliminaries on Clifford algebras and spinors induced on a spacelike surface in four-dimensional Minkowski space; we systematically use the complex quaternions to describe the various algebraic tools. Section \ref{sec fundamental eq} deals with the relation between spinors fields induced on a spacelike surface and the equations of Gauss, Codazzi and Ricci satisfied by the second fundamental form of the surface; the proof of the  main technical lemma is given Section \ref{secpflem}. In Section \ref{sec representation}  we give the representation formula of the immersion by the spinor field, which may be considered as an extension to general spacelike surfaces of the Weierstrass representation of maximal surfaces. In Section \ref{sec hypersurfaces}, we focus our attention to the special cases of immersions in $\R^3,$ $\mathbb{H}^3$ and $\R^{1,2}.$ In Section \ref{sec flat} we study flat surfaces with flat normal bundle and regular Gauss map: we first study the Gauss map of  a spacelike surface, in the general case, and in the case where the surface is flat with flat normal bundle (Section \ref{ss gauss map}); we then give a description of flat surfaces with flat normal bundle and regular Gauss map using spinors (the analysis of the problem is done Section \ref{spinor description} and the synthesis Section \ref{ss recovering}), and we finally describe the flat surfaces in three-dimensional hyperbolic space $\mathbb{H}^3.$ 
\section{Preliminaries}\label{preliminaries}
In this section, we first give algebraic preliminaries about Clifford algebras and spinors in four-dimensional Minkowski space $\R^{1,3}$, then quote the basic properties of spinors induced on a spacelike surface of $\R^{1,3},$ and finally describe the corresponding abstract constructions, starting with an abstract Riemannian surface and a Lorentzian bundle of rank 2 on the surface.
\subsection{Spinors of the four-dimensional Minkowski space}
\subsubsection{The Clifford algebra, the Spin group and the spinorial representation}\label{basic facts}
The Minkowski space $\R^{1,3}$ is the space $\R^4$ endowed with the metric
$$g=-dx_0^2+dx_1^2+dx_2^2+dx_3^2.$$
Using the Clifford map 
\begin{eqnarray}
\R^{1,3}&\rightarrow&\HC(2)\label{Clifford map}\\
(x_0,x_1,x_2,x_3)&\mapsto& \left(\begin{array}{cc}0&ix_0\1+x_1I+x_2J+x_3K\\-ix_0\1+x_1I+x_2J+x_3K&0\end{array}\right)\nonumber
\end{eqnarray}
where $\HC(2)$ stands for the set of $2\times 2$ matrices with entries belonging to the set of complex quaternions $\HC,$ the Clifford algebra of $\R^{1,3}$ is
$$Cl(1,3)=\left\{\left(\begin{array}{cc}a&b\\\tch{b}&\tch{a}\end{array}\right),\ a,b\in\HC\right\}$$
where, if $\xi=q_0\1+q_1I+q_2J+q_3K$ belongs to $\HC,$ we denote  
$$\tch{\xi}:=\overline{q_0}\1+\overline{q_1}I+\overline{q_2}J+\overline{q_3}K.$$
The Clifford sub-algebra of elements of even degree is
\begin{equation}\label{identification Cl0}
Cl_0(1,3)=\left\{\left(\begin{array}{cc}a&0\\0&\tch{a}\end{array}\right),\ a\in\HC\right\}\simeq\HC.
\end{equation}
Consider the bilinear map $H:\HC\times\HC\rightarrow\C$ defined by
\begin{equation}\label{H HC}
H(\xi,\xi')=q_0q_0'+q_1q_1'+q_2q_2'+q_3q_3'
\end{equation}
where $\xi=q_0\1+q_1I+q_2J+q_3K$ and $\xi'=q_0'\1+q_1'I+q_2'J+q_3'K.$ It is $\C$-bilinear for the natural complex structure $i$ on $\HC.$ Its real part, denoted by $\langle.,.\rangle,$ is a real scalar product of signature $(4,4)$ on $\HC.$ We consider
$$Spin(1,3):=\{q\in\HC:\ H(q,q)=1\}\hspace{.5cm}\subset \hspace{.5cm}Cl_0(1,3).$$
Using the identification 
\begin{eqnarray*}
\R^{1,3}&\simeq&\{ix_0\1+x_1I+x_2J+x_3K,\ (x_0,x_1,x_2,x_3)\in\R^4\}\\
&\simeq&\{\xi\in\HC:\ \xi=-\tch{\overline{\xi}}\},
\end{eqnarray*}
 where, if $\xi=q_0\1+q_1I+q_2J+q_3K$ belongs to $\HC,$ we denote  $\overline{\xi}=q_0\1-q_1I-q_2J-q_3K,$ we get the double cover
\begin{eqnarray}
\Phi:Spin(1,3)&\stackrel{2:1}{\longrightarrow}& SO(1,3)\label{double cover}\\
q&\mapsto &(\xi\in\R^{1,3}\mapsto q\xi \tch{q}^{-1}\in\R^{1,3}).\nonumber
\end{eqnarray}
Here and below $SO(1,3)$ stands for the component of the identity of the group of Lorentz transformations of $\R^{1,3}.$ Let us denote by $\rho:Cl(1,3)\rightarrow End(\HC)$ the complex representation of $Cl(1,3)$ on $\HC$ given by
\begin{eqnarray*}
\rho\left(\begin{array}{cc}a&b\\\tch{b}&\tch{a}\end{array}\right):&&\xi\in\HC\simeq\left(\begin{array}{c}\xi\\\tch{\xi}\end{array}\right)\mapsto\left(\begin{array}{cc}a&b\\\tch{b}&\tch{a}\end{array}\right)\left(\begin{array}{c}\xi\\\tch{\xi}\end{array}\right)\simeq a\xi+b\tch{\xi}\in\HC,
\end{eqnarray*}
where the complex structure on $\HC$ is given here by the multiplication by $I$ on the right. The spinorial representation of $Spin(1,3)$ is the restriction to $Spin(1,3)$ of the representation $\rho$ and reads
\begin{eqnarray*}
\rho_{|Spin(1,3)}:\ Spin(1,3)&\rightarrow &End_{\C}(\HC)\\
a&\mapsto &(\xi\in\HC\mapsto a\xi\in\HC).
\end{eqnarray*}
This representation splits into
$$\HC=S^+\oplus S^-$$
where $S^+=\{\xi\in\HC:\ \xi I=i\xi\}$ and $S^-=\{\xi\in\HC:\ \xi I=-i\xi\}.$ Explicitly, we have
$$S^+=(\C\oplus\C J)(\1-iI)\hspace{1cm}\mbox{and}\hspace{1cm}S^-=(\C\oplus\C J)(\1+iI).$$
Note that, if $(e_0,e_1,e_2,e_3)$ stands for the canonical basis of $\R^{1,3},$ the complexified volume element $i\ e_0\cdot e_1\cdot e_2\cdot e_3$ acts as $+Id$ on $S^+$ and as $-Id$ on  $S^-.$ 

We finally observe that $Spin(1,3)$ also naturally acts on $\HC$ by multiplication on the right
\begin{eqnarray*}
Spin(1,3)&\rightarrow&(\HC\rightarrow\HC)\\
q&\mapsto&(\xi\mapsto \xi q),
\end{eqnarray*}
and that the following property holds:
\begin{equation}\label{property right action}
X\cdot(\xi q)=(X\cdot\xi)\tch{q}
\end{equation}
for all $X\in \R^{1,3}\subset Cl_1(1,3),\ \xi\in\HC$ and $q\in Spin(1,3);$ this is because the action of $X$ on $\xi q$ is given by the action on  $\left(\begin{array}{cc}\xi q\\\tch{\xi}\tch{q}\end{array}\right)$ of the matrix representation (\ref{Clifford map}). 

\subsubsection{Spinors under the splitting $\R^{1,3}=\R^{1,1}\times\R^2$}\label{preliminaries splitting}
We now consider the splitting $\R^{1,3}=\R^{1,1}\times\R^2$ and the corresponding inclusion $SO(1,1)\times SO(2)\subset SO(1,3).$ By using the very definition (\ref{double cover}) of $\Phi,$ we easily get
\begin{eqnarray}
\Phi^{-1}(SO(1,1)\times SO(2))&=&\{\cos z +\sin z I,\ z\in\C\}=:S^1_{\C}\hspace{.3cm}\subset Spin(1,3);
\end{eqnarray}
more precisely, setting $z=\theta+i\varphi,$ $\theta,\varphi\in\R,$ we have, in $\HC,$
$$\cos z+\sin z I=(\cosh\varphi+i\sinh\varphi I).(\cos\theta+\sin\theta I),$$
and $\Phi(\cos z +\sin z I)$ is the Lorentz transformation of $\R^{1,3}$ which consists of a Lorentz transformation of angle $2\varphi$ in $\R^{1,1}$ and a rotation of angle $2\theta$ in $\R^2$.
Thus, defining
\begin{equation}\label{def spin1,1}
Spin(1,1):=\{\pm\left(\cosh\varphi+i\sinh\varphi I\right),\ \varphi\in\R\}\subset Spin(1,3)
\end{equation}
and
\begin{equation}\label{def spin2}
Spin(2):=\{\cos\theta+\sin\theta I,\ \theta\in\R\}\subset Spin(1,3),
\end{equation}
we have 
$$S^1_{\C}=Spin(1,1).Spin(2)\simeq Spin(1,1)\times Spin(2)/\mathbb{Z}_2$$
and the double cover
$$\Phi:\hspace{.5cm}S^1_{\C}\hspace{.3cm}\stackrel{2:1}{\longrightarrow}\hspace{.3cm} SO(1,1)\times SO(2).$$
 If we now restrict the spinorial representation $\rho$ of $Spin(1,3)$ to $S^1_{\C}\subset Spin(1,3),$ the representation in $\HC=S^+\oplus S^-$ splits into the sum of four complex lines: 
\begin{equation}\label{splitting HC}
S^+=S^{++}\oplus S^{--}\hspace{1cm}\mbox{ and }\hspace{1cm}S^-=S^{+-}\oplus S^{-+}
\end{equation}
where 
$$S^{++}=\C J(\1-iI),\hspace{.2cm} S^{--}=\C(\1-iI),\hspace{.2cm}S^{+-}=\C(\1+iI)\hspace{.2cm}\mbox{and}\hspace{.2cm}S^{-+}=\C J(\1+iI).$$
Note that $e_0\cdot e_1$ acts as $+Id$ on $S^{++}$ and on $S^{+-},$ and as $-Id$ on $S^{--}$ and on $S^{-+},$ whereas $i\ e_2\cdot e_3$ acts as $+Id$ on $S^{++}$ and on $S^{-+},$ and as $-Id$ on $S^{--}$ and on $S^{+-}.$
\begin{rem}\label{rep isom tensor}
It is not difficult to show that the representation
\begin{eqnarray}
Spin(1,1)\times Spin(2)&\rightarrow& End_{\C}(\HC)\label{rep spin11spin2}\\
(g_1,g_2)&\mapsto&\rho(g):\xi\mapsto g\xi,\nonumber
\end{eqnarray}
where $g=g_1g_2\in S^1_{\C}=Spin(1,1).Spin(2),$  is equivalent to the representation 
\begin{equation}\label{splitting rep spin}
\rho_1\otimes\rho_2=\rho_1^+\otimes\rho_2^+\ \oplus\ \rho_1^-\otimes\rho_2^-\ \oplus\ \rho_1^+\otimes\rho_2^-\ \oplus\ \rho_1^-\otimes\rho_2^+
\end{equation}
of $Spin(1,1)\times Spin(2),$ where $\rho_1=\rho_1^++\rho_1^-$ and $\rho_2=\rho_2^++\rho_2^-$ are the spinorial representations of $Spin(1,1)$ and $Spin(2);$ moreover, the decomposition (\ref{splitting HC}) corresponds to the decomposition (\ref{splitting rep spin}). Indeed, the restrictions of the representation (\ref{rep spin11spin2}) to the complex lines $S^{++},$ $S^{--},$ $S^{+-}$ and $S^{-+}$ are respectively equivalent to the multiplications by $e^{-iz},$ $e^{iz},$ $e^{i\overline{z}}$ and $e^{-i\overline{z}}$ on $\C,$ with $z=\theta+i\varphi,$ where $\varphi\in\R$ describes the $Spin(1,1)-$factor and $\theta\in\R/2\pi\Z$ the $Spin(2)-$factor of $Spin(1,1)\times Spin(2),$ as in (\ref{def spin1,1})-(\ref{def spin2}). 
\end{rem}

\subsection{Basic properties of spinors induced on a spacelike surface in $\R^{1,3}$}
Let $(M^2,g)$ be an oriented Riemannian surface isometrically immersed into the four-dimensio\-nal Minkowski space $\R^{1,3}$. Let us denote by $E$ its normal bundle and by $B:TM\times TM\lgra E$ its second fundamental form defined by
$$B(X,Y)=\overline{\nabla}_XY-\nabla_XY,$$
where $\nabla$ and $\nablab$ are the Levi-Civita connections of $M$ and $\R^{1,3}$ respectively. We moreover assume that $E$ is oriented, in space and in time. The second fundamental form satisfies the equations
\be
\item $K=\left<B(e_2,e_2),B(e_3,e_3)\right>-|B(e_2,e_3)|^2 $  (Gauss equation), 
\item $K_N=\left<(S_{e_0}\circ S_{e_1}-S_{e_1}\circ S_{e_0})(e_2),e_3\right> $ (Ricci equation), 
\item  $(\widetilde{\nabla}_XB)(Y,Z)-(\widetilde{\nabla}_YB)(X,Z)=0$ (Codazzi equation),
\ee
where $K$ and $K_N$ are the curvatures of $(M,g)$ and $E,$ $(e_2,e_3)$ and $(e_0,e_1)$ are orthonormal, space- and time-oriented bases of $TM$ and $E$ respectively, and where $\widetilde{\nabla}$ is the natural connection induced on $T^*M^{\otimes 2}\otimes E$. As usual, if $\nu\in E,$ $S_{\nu}$ stands for the symmetric operator on $TM$ such that $$\langle S_{\nu}(X),Y\rangle=\langle B(X,Y),\nu\rangle$$ for all $X,Y\in TM.$
\begin{rem}\label{rmq th fal}
Suppose that $E$ is a bundle of rank 2 on a Riemannian surface $(M,g),$ equipped with a Lorentzian metric and a compatible connection, and assume that $B:TM\times TM\rightarrow E$ is a bilinear map satisfying the equations (1), (2) and (3) above; the fundamental theorem says that there is locally a unique isometric immersion of $M$ in $\R^{1,3}$ with normal bundle $E$ and second fundamental form B. We will obtain a spinorial proof of this theorem, in Corollary \ref{spinorial proof fundamental theorem} below.
\end{rem}
It is well known that there is an identification between the spinor bundle of $\R^{1,3}$ over $M,$ $\Sigma\R^{1,3}_{|M},$ and the spinor bundle of $M$ twisted by the normal bundle $\Sigma=\Sigma E\otimes\Sigma M$; see  \cite{Ba,HZ}, and also Remark \ref{rep isom tensor}. Moreover, we have the spinorial Gauss formula: for any $\varphi\in\Sigma$ and any $X\in TM$,
$$\overline\nabla_X\varphi=\nabla_X\varphi+\frac{1}{2}\sum_{j=2,3}e_j\cdot B(X,e_j) \cdot\varphi,$$
where $\overline\nabla$ is the spinorial connection of $\Sigma\R^{1,3},$ $\nabla$ is the spinorial connection of $\Sigma$ defined by
$$\nabla=\ \ \nabla^{\Sigma E}\otimes\iid_{\Sigma M}\ \ +\ \ \iid_{\Sigma E}\otimes\nabla^{\Sigma M}$$
and the dot $"\cdot"$ is the Clifford action in $\R^{1,3};$ see \cite{Ba,Va}. Therefore, if $\varphi$ is a parallel spinor of $\R^{1,3}$, then its restriction to $M$ satisfies
\beqt\label{eqkilling}
\nabla_X\varphi=-\frac{1}{2}\sum_{j=2,3}e_j\cdot B(X,e_j) \cdot\varphi,
\eeqt
and by taking the trace the following Dirac equation
\beqt\label{eqdirac}
D\varphi=\vec{H}\cdot\varphi,
\eeqt
where $D\varphi:=e_2\cdot\nabla_{e_2}\varphi+e_3\cdot\nabla_{e_3}\varphi$ and where $\displaystyle{\vec H=\frac{1}{2}\ tr_gB}$ is the mean curvature vector of $M$ in $\R^{1,3}.$
\\

Finally, correspondingly to the splittings (\ref{splitting HC})-(\ref{splitting rep spin}) we have $\Sigma=\Sigma^+\oplus\Sigma^-$ with
$$\Sigma^+=\Sigma^{++}\oplus\Sigma^{--}\hspace{1cm}\mbox{and}\hspace{1cm}\Sigma^-=\Sigma^{+-}\oplus\Sigma^{-+},$$
where 
$$\Sigma^{++}=\Sigma^{+}E\otimes\Sigma^{+}M,\hspace{.5cm} \Sigma^{--}=\Sigma^{-}E\otimes\Sigma^{-}M,\hspace{.5cm} \Sigma^{+-}=\Sigma^{+}E\otimes\Sigma^{-}M$$
and
$$\Sigma^{-+}=\Sigma^{-}E\otimes\Sigma^{+}M.$$
\subsection{Twisted spinor bundle}\label{twisted spinor bundle}
Let $(M^2,g)$ be an oriented Riemannian surface and $E$ a space- and time-oriented Lorentzian bundle of rank 2 over $M$ equipped with a compatible connection, with given spin structures. We consider the spinor bundle $\Sigma$ over $M$ twisted by $E$ and defined by
$$\Sigma:=\Sigma E\otimes\Sigma M.$$
We endow $\Sigma$ with the spinorial connection $\nabla$ defined by
$$\nabla=\nabla^{\Sigma E}\otimes\iid_{\Sigma M}+\iid_{\Sigma E}\otimes\nabla^{\Sigma M}.$$
We also define the Clifford product $\cdot$ by
$$\left\{\begin{array}{l} X\cdot\varphi=(X\cdot_{_E}\alpha)\otimes\sigma\quad\text{if}\ X\in\Gamma(E)\\ \\
X\cdot\varphi=\overline\alpha\otimes(X\cdot_{_M}\sigma)\quad\text{if}\ X\in\Gamma(TM)
\end{array}
\right.$$
where $\varphi=\alpha\otimes\sigma$ belongs to $\Sigma,$ $\cdot_{_E}$ and $\cdot_{_M}$ denote the Clifford actions on $\Sigma E$ and $\Sigma M$ and $\overline{\alpha}=\alpha^+-\alpha^-$ $\in$ $\Sigma E=\Sigma^+E\oplus\Sigma^-E.$

The twisted spinor bundle $\Sigma$ is naturally a vector bundle on $\HC$ (on the right): let us denote by $\rho_1=\rho_1^++\rho_1^-$ and $\rho_2=\rho_2^++\rho_2^-$ the spinorial representations of $Spin(1,1)$ and $Spin(2).$ Recall first that $\Sigma$ is the vector bundle associated to the principal bundle 
\begin{equation}\label{def Q tilde}
\tilde{Q}:=\tilde{Q}_1\times_M \tilde{Q}_2=\{(\tilde{s_1},\tilde{s_2})\in\tilde{Q_1}\times\tilde{Q_2}:\ p_1(\tilde{s_1})=p_2(\tilde{s_2})\}
\end{equation}
and to the representation $\rho_1\otimes\rho_2$ of the structure group $Spin(1,1)\times Spin(2),$ where, if $Q_1$ and $Q_2$ denote the $SO(1,1)$ and $SO(2)$ principal bundles of the oriented and orthonormal frames  of $E$ and $TM,$ the double coverings $\tilde{Q_1}\rightarrow Q_1$ and $\tilde{Q_2}\rightarrow Q_2$ are the given spin structures on $E$ and $TM$ and the maps $p_1:\tilde{Q_1}\rightarrow M$ and $p_2:\tilde{Q_2}\rightarrow M$ are the corresponding projections. We noted Remark \ref{rep isom tensor} that the representation $\rho_1\otimes\rho_2$ is equivalent to the representation
\begin{eqnarray}\label{rep regular}
Spin(1,1)\times Spin(2)&\rightarrow& End_{\C}(\HC)\\
(g_1,g_2)&\mapsto&\rho(g):\xi\mapsto g\xi\nonumber
\end{eqnarray}
where $g=g_1g_2\in S^1_{\C}=Spin(1,1).Spin(2).$ Obviously, the maps  $\xi\mapsto g\xi$ belong in fact to $End_{\HC}(\HC),$ the space of endomorphisms of $\HC$ which are $\HC-$linear, where the linear structure on $\HC$ is given by the multiplication on the right: $\Sigma$ is thus naturally equipped with a linear right-action of $\HC.$

Since the group $S^1_{\C}=Spin(1,1).Spin(2)$ belongs to $Spin(1,3),$ which preserves the  complex bilinear map $H$ defined on $\HC$ by (\ref{H HC}), the bundle $\Sigma$ is also equipped with a complex bilinear map $H$ and with a real scalar product $\langle.,.\rangle:=\pre\ H$ of signature $(4,4).$ We note the following properties: $H$ vanishes on the bundles $\Sigma^+$ and $\Sigma^-$ (and thus also on the four bundles $\Sigma^{\pm\pm}$) since $H$ vanishes on $S^+$ and $S^-;$ on the other hand, the real scalar product $\langle.,.\rangle$ has signature $(2,2)$ on the two bundles $\Sigma E\otimes\Sigma^+M$ and $\Sigma E\otimes\Sigma^-M,$ and these bundles are moreover orthogonal with respect to $H.$ These last properties are consequences of the fact that these bundles are respectively associated to the representations of $S^1_{\C}$ on $S^{++}\oplus S^{-+}=\C J\oplus\C K$ and on $S^{+-}\oplus S^{--}=\C \1\oplus\C I$ given by the multiplication on the left; see Section \ref{preliminaries splitting}. 

We may also define a $\HC$-valued scalar product on $\Sigma$ by 
\begin{equation}\label{def scalar HC}
\langle\langle\psi,\psi'\rangle\rangle:=\overline{\xi'}\xi\in\HC,
\end{equation}
where $\xi$ and $\xi'\in\HC$ are respectively the components of $\psi$ and $\psi'$ in some local section of $\tilde{Q}$  (recall that, if $\xi'=q_0'\1+q_1'I+q_2'J+q_3'K$ belongs to $\HC,$ we define $\overline{\xi'}=q_0'\1-q_1'I-q_2'J-q_3'K$). This scalar product satisfies the following properties: for all $\psi,\psi'\in\Sigma$ and for all $X\in E\oplus TM,$
\begin{equation}\label{properties pairing}
\langle\langle\psi,\psi'\rangle\rangle=\overline{\langle\langle\psi',\psi\rangle\rangle}\hspace{1cm}\mbox{and}\hspace{1cm}\langle\langle X\cdot\psi,\psi'\rangle\rangle=-\bigtch{\langle\langle\psi,X\cdot\psi'\rangle\rangle}.
\end{equation}
We stress that the product $\langle\langle.,.\rangle\rangle$ is not $\C-$bilinear with respect to the standard complex structure on $\Sigma$ (which corresponds to the right multiplication by $I$ on $\HC$) but is $\C-$bilinear if $\Sigma$ is endowed with the complex structure given by the Clifford action of $-e_0\cdot e_1\cdot e_2\cdot e_3$ (which corresponds to the multiplication by $i$ on $\HC$).

Finally, note that, by definition, $H(\psi,\psi')$ is the coefficient of $\1$ in the decomposition of $\langle\langle\psi,\psi'\rangle\rangle$ in the basis $\1, I,J,K$ of $\HC,$ and that (\ref{properties pairing}) yields
\begin{equation}\label{H clifford}
H(\psi,\psi')=H(\psi',\psi)\hspace{1cm}\mbox{and}\hspace{1cm}H(X\cdot\psi,\psi')=-\overline{H(\psi,X\cdot\psi')}
\end{equation}
for all $\psi,\psi'\in\Sigma$ and for all $X\in E\oplus TM.$

\subsection{Notation}\label{notation generale}
Throughout the paper, we will use the following notation: if $\tilde{s}\in\tilde{Q}$ is a given frame, the brackets $[\cdot]$ will denote the coordinates $\in \HC$ of the spinor fields in $\tilde{s},$ that is, for $\varphi\in\Sigma,$
$$\varphi\simeq[\tilde{s},[\varphi]]\hspace{1cm}\in\hspace{.5cm} \Sigma\simeq\tilde{Q}\times\HC/\rho_1\otimes\rho_2.$$
  We will also use the brackets to denote the coordinates in $\tilde{s}$ of the elements of the Clifford algebra $Cl(E\oplus TM)$: $X\in Cl_0(E\oplus TM)$ and  $Y\in Cl_1(E\oplus TM)$ will be respectively represented by $[X]$ and $[Y]\in\HC$ such that, in $\tilde{s},$ 
$$X\simeq \left(\begin{array}{cc}[X]&0\\0&\tch{[X]}\end{array}\right)\hspace{1cm}\mbox{and}\hspace{1cm}Y\simeq\left(\begin{array}{cc}0&[Y]\\\tch{[Y]}&0\end{array}\right).$$
Note that
$$[X\cdot\varphi]=[X][\varphi]\hspace{1cm}\mbox{and}\hspace{1cm}[Y\cdot\varphi]=[Y]\tch{[\varphi]}.$$
If $(e_0,e_1)$ and $(e_2,e_3)$ are positively oriented and orthonormal frames of $E$ and $TM,$ a frame $\tilde{s}\in\tilde{Q}$ such that $\pi(\tilde{s})=(e_0,e_1,e_2,e_3),$ where $\pi:\tilde{Q}\rightarrow Q_1\times_M Q_2$ is the natural projection onto the bundle of the orthonormal frames of $E\oplus TM$, will be called \textit{adapted to the frame} $(e_0,e_1,e_2,e_3);$ in such a frame, $e_0,$ $e_1,$ $e_2$ and $e_3\in Cl_1(E\oplus TM)$ are respectively represented by $i\textit{1},$ $I,$ $J$ and $K\in\HC.$

\section{Spinor field and fundamental equations}\label{sec fundamental eq}
We prove here that a spinor field solution of the Dirac equation gives rise to a bilinear map satisfying the equations of Gauss, Codazzi and Ricci. As in the paper \cite{Fr} (and afterwards in \cite{La,LR,Mo, Ro} and \cite{BLR}), the proof relies on the fact that such a spinor field necessarily solves a Killing type equation (equation (\ref{eqnnablaphi}) in Lemma \ref{lem1} below).  

\begin{thm}\label{thm1}
Let $(M^2,g)$ be an oriented surface and $E$ an oriented Lorentzian vector bundle of rank $2$ on $M$ equipped with a compatible connection, with given spin structures. Let $\Sigma=\Sigma E\otimes\Sigma M$ be the twisted spinor bundle and $D$ its Dirac operator. Let $\vec{H}$ be a section of $E$ and assume that there exists a spinor $\varphi\in\Gamma(\Sigma)$ solution of 
\begin{equation}\label{dirac equation}
D\varphi=\vec{H}\cdot\varphi
\end{equation} 
with $H(\varphi,\varphi)=1$. Then the bilinear map 
$$B:TM\times TM  \longrightarrow E$$
defined by
\begin{equation}\label{B function phi}
\left<B(X,Y),\nu\right>=2\left<X\cdot\nabla_Y\varphi,\nu\cdot\varphi\right>
\end{equation}
for all $X,Y\in \Gamma(TM)$ and all $\nu\in \Gamma(E)$ is symmetric, satisfies the Gauss, Codazzi and Ricci equations and is such that
$$\vec{H}=\frac{1}{2}\trace{B}.$$
\end{thm}
\noindent
In the theorem and below,  we use the same notation $\left<.,.\right>$ to denote the scalar products on $TM,$ on $E,$ and on $\Sigma$ (the real scalar product, defined above, Section \ref{twisted spinor bundle}). In order to prove this theorem, we state the following lemma.
\begin{lem}\label{lem1}
If $\varphi$ is a solution of the Dirac equation (\ref{dirac equation}) then, for all $X\in \Gamma(TM)$,
\begin{equation}\label{eqnnablaphi}
\nabla_X\varphi=\eta(X)\cdot\varphi,
\end{equation}
with 
\begin{equation}\label{relation eta B}
\eta(X)=-\frac{1}{2}\sum_{j=2}^3e_j\cdot B(e_j,X),
\end{equation}
where $B$ is the bilinear map defined above.
\end{lem}
\noindent
Using this lemma, we will prove Theorem \ref{thm1}. The proof of this lemma will be given in Section \ref{secpflem}. \\ \\
{\it Proof of Theorem \ref{thm1}:} 
The equations of Gauss, Codazzi and Ricci appear to be the integrability conditions of (\ref{eqnnablaphi}).  We compute the spinorial curvature of $\varphi$. A direct computation yields
\beqt\label{curv}
\mathcal{R}(X,Y)\varphi=d^{\nabla}\eta(X,Y)\cdot\varphi+\big(\eta(Y)\cdot\eta(X)-\eta(X)\cdot\eta(Y)\big)\cdot\varphi,
\eeqt
where
$$d^{\nabla}\eta(X,Y):=\nabla_X(\eta(Y))-\nabla_Y(\eta(X))-\eta([X,Y]).$$
Here and below we also denote by $\nabla$ the natural connexion on the Clifford bundle $Cl(E\oplus TM)\simeq Cl(E)\hat{\otimes}Cl(M).$
\begin{lem}\label{lem calculs courbure}
Let $(e_0,e_1)$ and $(e_2,e_3)$ be positively oriented and orthonormal bases of $E$ and $TM$ respectively. We have:
\be
\item The left-hand side of \eqref{curv} satisfies
$$\mathcal{R}(e_2,e_3)\varphi=-\frac{1}{2}Ke_2\cdot e_3\cdot\varphi-\frac{1}{2}K_Ne_0\cdot e_1\cdot\varphi.$$
\item The first term of the right-hand side of \eqref{curv} satisfies
 $$d^{\nabla}\eta(X,Y)=-\frac{1}{2}\sum_{j=2}^3e_j\cdot\Big((\widetilde{\nabla}_XB)(Y,e_j)-(\widetilde{\nabla}_YB)(X,e_j)\Big).$$
\item The second term of the right-hand side of \eqref{curv} satisfies
\beQ
\eta(e_3)\cdot\eta(e_2)-\eta(e_2)\cdot\eta(e_3)&=&-\frac{1}{2}\big(\left<B(e_2,e_2),B(e_3,e_3)\right>-|B(e_2,e_3)|^2\big)e_2\cdot e_3\\
&&-\frac{1}{2}\left<\left(S_{e_0}\circ S_{e_1}-S_{e_1}\circ S_{e_0}\right)(e_2),e_3\right>e_0\cdot e_1.\eeQ
\ee
\end{lem}
{\it Proof:} We first compute $\mathcal{R}(e_2,e_3)\varphi$. We recall that $\Sigma=\Sigma E\otimes\Sigma M$ and assume that $\varphi=\alpha\otimes\sigma$ with $\alpha\in\Sigma E$ and $\sigma\in\Sigma M$. Thus,
$$\mathcal{R}(e_2,e_3)\varphi=\mathcal{R}^E(e_2,e_3)\alpha\otimes\sigma+\alpha\otimes\mathcal{R}^M(e_2,e_3)\sigma.$$
Moreover, by the Ricci identity on $M$, we have 
$$\mathcal{R}^M(e_2,e_3)\sigma=-\frac{1}{2}Ke_2\cdot e_3\cdot\sigma,$$
where $K$ is the Gauss curvature of $(M,g)$. Similarly, we have
$$\mathcal{R}^E(e_2,e_3)\alpha=-\frac{1}{2}K_Ne_0\cdot e_1\cdot\alpha,$$
where $K_N$ is the curvature of the connection on $E.$ These last two relations give the first claim of the lemma. For the second claim of the lemma, we choose $e_j$ so that at $p\in M$, $\nabla {e_j}_{|p}=0$. Then, we have
\beQ
d^{\nabla}\eta(X,Y)&=&\nabla_X(\eta(Y))-\nabla_Y(\eta(X))-\eta([X,Y])\\
&=&\sum_{j=2,3}-\frac{1}{2}\nabla_X\big(e_j\cdot B(Y,e_j)\big)+\frac{1}{2}\nabla_Y\big(e_j\cdot B(X,e_j)\big)+\frac{1}{2}e_j\cdot B([X,Y],e_j)\\
&=&\sum_{j=2,3}-\frac{1}{2}e_j\cdot(\widetilde{\nabla}_XB)(Y,e_j)+\frac{1}{2}e_j\cdot(\widetilde{\nabla}_YB)(X,e_j)+\frac{1}{2}e_j\cdot B([X,Y],e_j)\\
&&\hspace{.5cm}-\frac{1}{2}e_j\cdot B(\nabla_XY,e_j)+\frac{1}{2}e_j\cdot B(\nabla_YX,e_j)\\
&=&\sum_{j=2,3}-\frac{1}{2}e_j\cdot\Big((\widetilde{\nabla}_XB)(Y,e_j)-(\widetilde{\nabla}_YB)(X,e_j)\Big)
\eeQ
since $\nabla_XY-\nabla_YX=[X,Y].$ We now prove the third assertion of the lemma. In order to simplify the notation, we set $B(e_i,e_j)=B_{ij}$. We have
\beQ
\eta(e_3)\cdot\eta(e_2)-\eta(e_2)\cdot\eta(e_3)&=&\frac{1}{4}\sum_{j,k=2}^{3}e_j\cdot B_{3j}\cdot e_k\cdot B_{2k}-\frac{1}{4}\sum_{j,k=2}^{3}e_j\cdot B_{2j}\cdot e_k\cdot B_{3k}\eeQ
\beQ
&=&\frac{1}{4}\Bigg[e_2\cdot B_{32}\cdot e_2\cdot B_{22}+e_2\cdot B_{32}\cdot e_3\cdot B_{23}+e_3\cdot B_{33}\cdot e_2\cdot B_{22}+e_3\cdot B_{33}\cdot e_3\cdot B_{23}\\
&&-e_2\cdot B_{22}\cdot e_2\cdot B_{32}-e_2\cdot B_{22}\cdot e_3\cdot B_{33}-e_3\cdot B_{23}\cdot e_2\cdot B_{32}-e_3\cdot B_{23}\cdot e_3\cdot B_{33}\Bigg]\\
&=&-\frac{1}{2}\Big[\left<B_{22},B_{33}\right>-|B_{23}|^2\Big]e_2\cdot e_3-\frac{1}{4}\Big[B_{22}\cdot B_{23}-B_{23}\cdot B_{22}+B_{23}\cdot B_{33}-B_{33}\cdot B_{23}\Big].
\eeQ
Now, if we write $B_{ij}=B_{ij}^0e_0+B_{ij}^1e_1$, and setting
$$A:=B_{22}\cdot B_{23}-B_{23}\cdot B_{22}+B_{23}\cdot B_{33}-B_{33}\cdot B_{23},$$
we have
\beQ
A&=&2\Big(B_{22}^0B_{23}^1-B_{23}^0B_{22}^1+B_{23}^0B_{33}^1-B_{33}^0B_{23}^1\Big)e_0\cdot e_1\\
&=&2\left<\left(S_{e_0}\circ S_{e_1}-S_{e_1}\circ S_{e_0}\right)(e_2),e_3\right>e_0\cdot e_1,
\eeQ
since for $k\in\{2,3\}$, we have 
$$S_{e_0}e_k=-B_{k2}^0e_2-B_{k3}^0e_3\hspace{1cm}\mbox{and}\hspace{1cm}S_{e_1}e_k=B_{k2}^1e_2+B_{k3}^1e_3.$$
$\finpreuve$\\
The following lemma permits to finish the proof:
\begin{lem}\label{order2}
If $T=\sum_{0\leq i<j\leq 3}t_{ij}e_i\cdot e_j$ is an element of order 2 of $\mathcal{C}l(E)\hat{\otimes}\mathcal{C}l(M)$ so that $T\cdot\varphi=0,$ where $\varphi\in\Gamma(\Sigma)$ is a spinor field such that $H(\varphi,\varphi)\neq 0,$ then $T=0$.
\end{lem}
\noindent
{\it Proof:} 
Recalling the notation introduced in Section \ref{notation generale}, in some local section $\tilde{s}$ of $\tilde{Q},$ $T$ is represented by $q\in\HC,$ and $\varphi$ by $\xi\in\HC;$ thus $T\cdot\varphi$ is represented by $q\xi.$ If $T\cdot\varphi=0,$ we have $q\xi=0;$ since $\xi$ is invertible in $\HC$ ($\overline{\xi}\xi=H(\varphi,\varphi)\neq 0$), we get $q=0,$ and the result. 
$\finpreuve$\\
We deduce from \eqref{curv}, Lemma \ref{lem calculs courbure} and Lemma \ref{order2} that
$$\left\{
\begin{array}{l}
K=\left<B(e_2,e_2),B(e_3,e_3)\right>-|B(e_2,e_3)|^2,\\ \\
K_N=\left<\left(S_{e_0}\circ S_{e_1}-S_{e_1}\circ S_{e_0}\right)(e_2),e_3\right>,\\ \\
(\widetilde{\nabla}_XB)(Y,e_j)-(\widetilde{\nabla}_YB)(X,e_j)=0,\quad\forall j=2,3,
\end{array}\right.$$
which are respectively Gauss, Ricci and Codazzi equations. $\finpreuve$\\
From Theorem \ref{thm1} and the fundamental theorem of submanifolds (see Remark \ref{rmq th fal}), we get that a spinor field solution of (\ref{dirac equation}) defines a local isometric immersion of $M$ into $\R^{1,3}$ with normal bundle $E$ and second fundamental form $B;$ we deduce the following corollary:
\begin{cor}\label{corollary first integration}
Let $(M^2,g)$ be an oriented surface and let $E$ be a space- and time-oriented Lorentzian vector bundle of rank $2$ on $M$ equipped with a compatible connection, with given spin structures. Let $\vec{H}$ be a section of $E$. The three following statements are equivalent.
\begin{enumerate}
\item There exists a local spinor field $\varphi\in\Gamma(\Sigma)$ with $H(\varphi,\varphi)=1$ solution of the Dirac equation
$$D\varphi=\vec{H}\cdot\varphi.$$
\item There exists a local spinor field $\varphi\in\Gamma(\Sigma)$ with $H(\varphi,\varphi)=1$ solution of 
$$\nabla_X\varphi=-\frac{1}{2}\sum_{j=2,3}e_j\cdot B(X,e_j) \cdot\varphi, $$
where $B:TM\times TM\lgra E$ is bilinear and $\frac{1}{2}\trace(B)=\vec{H}.$
\item There exists a local (spacelike) isometric immersion of $(M^2,g)$ into $\R^{1,3}$ with normal bundle $E$, second fundamental form $B$ and mean curvature $\vec{H}$.
\end{enumerate}
The form $B$ and the spinor field $\varphi$ are linked by (\ref{B function phi}). Moreover, if $M$ is simply connected, the spinor field $\varphi$ and the isometric immersion in the equivalent statements (1), (2) and (3) are defined globally on $M.$ The first part of Theorem \ref{th main result} follows.
\end{cor}

\section{Proof of Lemma \ref{lem1}}\label{secpflem}
The main ideas of the proof amount to Friedrich \cite{Fr}. It will be convenient to consider here the complex structure $i:=-e_0\cdot e_1\cdot e_2\cdot e_3$ defined on the Clifford bundle $Cl(E\oplus TM)$ by the multiplication on the left, and on the spinor bundle $\Sigma$ by the Clifford action. The map $H:\Sigma\times\Sigma\rightarrow\C$ is $\C-$bilinear with respect to this complex structure (see Section \ref{twisted spinor bundle}), whereas the Clifford action of vectors belonging to $E\oplus TM$ is antilinear: it satisfies
$$i(x\cdot\varphi)=(ix)\cdot\varphi=-x\cdot(i\varphi)$$
for all $x\in E\oplus TM$ and all $\varphi\in\Sigma.$ In the decomposition
\begin{equation}\label{dec sigma}
\Sigma E\otimes\Sigma M\hspace{.5cm} =\hspace{.5cm} \Sigma E\otimes\Sigma^+M\hspace{.5cm} \oplus\hspace{.5cm} \Sigma E\otimes\Sigma^-M
\end{equation}
the spinor field $\varphi$ splits into
\begin{equation}\label{dec phi}
\varphi\hspace{.2cm}=\hspace{.2cm}\varphi^+\hspace{.2cm}+\hspace{.2cm}\varphi^-.
\end{equation}
Note that here, and in this section only, the decomposition (\ref{dec phi}) is not the decomposition in $\Sigma=\Sigma^+\oplus\Sigma^-,$ but really in (\ref{dec sigma}). We recall (see Section \ref{twisted spinor bundle}) that $H$ vanishes on the bundles $\Sigma^{\pm}E\otimes\Sigma^{\pm}M,$ and that the bundles $\Sigma E\otimes\Sigma^+M$ and  $\Sigma E\otimes\Sigma^-M$ are orthogonal with respect to $H.$ Let us define $h^+$ and $h^-$ in $\C$ such that
$${h^+}^2:=H(\varphi^+,\varphi^+)\hspace{1cm}\mbox{ and }\hspace{1cm}{h^-}^2:=H(\varphi^-,\varphi^-).$$
Since $H(\varphi,\varphi)={h^+}^2+{h^{-}}^2=1,$ we have $h^+$ or $h^-\neq 0;$ interchanging the roles of $h^+$ and $h^-$ if necessary, we may suppose without loss of generality that $h^-\neq 0.$ Let $(e_0,e_1)$ be a positively oriented basis of $E$ such that $e_0$ is future-directed; setting
$$e:=\frac{\sqrt{2}}{2}(e_0-e_1)\hspace{1cm}\mbox{and}\hspace{1cm}e':=\frac{\sqrt{2}}{2}(e_0+e_1),$$
we get 
\begin{equation}\label{properties e}
e\cdot e=e'\cdot e'=0\hspace{1cm}\mbox{and}\hspace{1cm}e\cdot e'-e_0\cdot e_1=1.
\end{equation} 
The spinors $e_2\cdot e\cdot\frac{\varphi^-}{h^-}$ and $e_2\cdot e'\cdot\frac{\varphi^-}{h^-}$ belong to $\Sigma E\otimes\Sigma^+M$ and, from (\ref{H clifford}) and (\ref{properties e}), satisfy 
\begin{equation*}
H\left(e_2\cdot e\cdot\frac{\varphi^-}{h^-},e_2\cdot e\cdot\frac{\varphi^-}{h^-}\right)=H\left(e_2\cdot e'\cdot\frac{\varphi^-}{h^-},e_2\cdot e'\cdot\frac{\varphi^-}{h^-}\right)=0
\end{equation*}
and
\begin{equation*}
H\left(e_2\cdot e\cdot\frac{\varphi^-}{h^-},e_2\cdot e'\cdot\frac{\varphi^-}{h^-}\right)=-1.
\end{equation*}
For the last identity we also use that $H(\varphi^-,e_0\cdot e_1\cdot\varphi^-)=0,$ which is in turn easily obtained using the decomposition 
$$\varphi^-=\varphi^{+-}+\varphi^{--}\hspace{.5cm}\in\hspace{.5cm}\Sigma^+E\otimes\Sigma^-M\ \ \oplus\ \ \Sigma^-E\otimes\Sigma^-M$$ 
together with $e_0\cdot e_1\cdot\varphi^{\pm -}=\pm\varphi^{\pm -}.$
In particular these spinors form a basis of $\Sigma E\otimes\Sigma^+M$ over $\C,$ and we have
\begin{equation*}
\nabla_X\varphi^+=-H\left(\nabla_X\varphi^+,e_2\cdot e'\cdot\frac{\varphi^-}{h^-}\right)e_2\cdot e\cdot\frac{\varphi^-}{h^-}-H\left(\nabla_X\varphi^+,e_2\cdot e\cdot\frac{\varphi^-}{h^-}\right)e_2\cdot e'\cdot\frac{\varphi^-}{h^-},
\end{equation*}
and thus, setting
\begin{equation}\label{def eta}
\eta(X)=-\frac{1}{{h^-}^2}\left\{H(\nabla_X\varphi^+,e_2\cdot e'\cdot\varphi^-)e_2\cdot e+H(\nabla_X\varphi^+,e_2\cdot e\cdot\varphi^-)e_2\cdot e'\right\},
\end{equation}
we get
\begin{eqnarray}\label{nabla phi+ eta p}
\nabla_X\varphi^+&=&\eta(X)\cdot\varphi^-.
\end{eqnarray}
Since
\begin{eqnarray*}
D\varphi^+&=&\vec{H}\cdot\varphi^-\\
&=&(e_2\cdot\eta(e_2)+e_3\cdot\eta(e_3))\cdot\varphi^-
\end{eqnarray*} 
and since $H(\varphi^-,\varphi^-)\neq 0,$ we obtain
\begin{equation}\label{eta H}
e_2\cdot\eta(e_2)+e_3\cdot\eta(e_3)=\vec{H};
\end{equation}
we use here a property similar to Lemma \ref{order2}, but for elements of odd degree in the Clifford algebra $Cl(E\oplus TM).$ Differentiating $H(\varphi^+,\varphi^+)+H(\varphi^-,\varphi^-)=1,$ we get
\begin{equation*}
0=H(\nabla_X\varphi^+,\varphi^+)+H(\nabla_X\varphi^-,\varphi^-)=H(\eta(X)\cdot\varphi^-,\varphi^+)+H(\nabla_X\varphi^-,\varphi^-).
\end{equation*}
Since $H(\eta(X)\cdot\varphi^-,\varphi^+)=-H(\varphi^-,\eta(X)\cdot\varphi^+)$ (by (\ref{H clifford})), we deduce
\begin{equation*}
H(\nabla_X\varphi^--\eta(X)\cdot\varphi^+,\varphi^-)=0,
\end{equation*}
i.e., since $(\varphi^-,e_0\cdot e_1\cdot\varphi^-)$ is a basis of $\Sigma E\otimes\Sigma^-M$ which is $H-$orthogonal,
\begin{equation*}
\nabla_X\varphi^--\eta(X)\cdot\varphi^+=\mu(X)e_0\cdot e_1\cdot\varphi^-
\end{equation*}
for some 1-form $\mu$ with complex values. Since $D\varphi^-=\vec{H}\cdot\varphi^+$ and by (\ref{eta H}) we get
\begin{equation*}
\left(\overline{\mu(e_2)}e_2\cdot e_0\cdot e_1+\overline{\mu(e_3)}e_3\cdot e_0\cdot e_1\right)\cdot\varphi^-=0
\end{equation*} 
and thus $\mu(e_2)=\mu(e_3)=0,$ i.e. $\mu=0.$ We thus get $\nabla_X\varphi^-=\eta(X)\cdot\varphi^+,$ which, together with (\ref{nabla phi+ eta p}), implies (\ref{eqnnablaphi}).
\\

We finally prove (\ref{relation eta B}): by (\ref{B function phi}) and since $\nabla_X\varphi=\eta(X)\cdot\varphi,$ we have
\begin{equation*}
\langle B(e_j,X),\nu\rangle=2\langle e_j\cdot\nabla_X\varphi,\nu\cdot\varphi\rangle
=2\langle e_j\cdot\eta(X)\cdot\varphi,\nu\cdot\varphi\rangle
\end{equation*}
for all $\nu\in E$ and for $j=2,3.$ We fix $X\in TM.$ By (\ref{def eta}), $\eta(X)$ may be written in the form
\begin{equation}\label{ecriture a priori eta}
\eta(X)=\eta_2e_2\cdot \nu_2+\eta_3e_3\cdot\nu_3,
\end{equation}
with $\eta_2,\eta_3\in\R$ and $\nu_2,\nu_3\in E.$ Thus
\begin{equation}\label{equation lemme B eta}
\langle B(e_2,X),\nu\rangle=-2\eta_2\langle\nu_2\cdot\varphi,\nu\cdot\varphi\rangle+2\eta_3\langle e_2\cdot e_3\cdot\nu_3\cdot\varphi,\nu\cdot\varphi\rangle.
\end{equation}
\begin{lem}
For all $\nu,\nu'\in E,$
$$\langle\nu\cdot\varphi,\nu'\cdot\varphi\rangle=\langle\nu,\nu'\rangle\hspace{1cm}\mbox{and}\hspace{1cm}\langle e_2\cdot e_3\cdot\varphi,\nu\cdot\nu'\cdot\varphi\rangle=0.$$
\end{lem}
\begin{proof}
To prove the first identity, we just observe that
$$\langle\nu\cdot\varphi,\nu'\cdot\varphi\rangle=-\langle\varphi,\nu\cdot\nu'\cdot\varphi\rangle=-\langle\nu'\cdot\varphi,\nu\cdot\varphi\rangle+2\langle\nu,\nu'\rangle$$
since $\nu\cdot\nu'=-\nu'\cdot\nu-2\langle\nu,\nu'\rangle$ and $\langle\varphi,\varphi\rangle=1.$ We now prove the second formula. Recall the definition (\ref{def Q tilde}) of $\tilde{Q},$ and denote by $\pi:\tilde{Q}\rightarrow Q_1\times_M Q_2$ the natural projection. In a local section $\tilde{s}$ of $\tilde{Q}$ such that $\pi(\tilde{s})=(e_0,e_1,e_2,e_3),$ $e_2\cdot e_3$ is represented by $I\in\HC$ and $\nu\cdot\nu'$ by $a\1+ibI,$ with $a,b\in\R;$ see Section \ref{notation generale}. Thus
$$\langle\langle e_2\cdot e_3\cdot\varphi,\nu\cdot\nu'\cdot\varphi\rangle\rangle=\overline{(a\1+ibI)[\varphi]}I[\varphi]=\overline{[\varphi]}(ib\1+aI)[\varphi]=ib\1+a\overline{[\varphi]}I[\varphi],$$
where $[\varphi]\in\HC$ represents $\varphi$ in $\tilde{s}.$ The last term $u:=a\overline{[\varphi]}I[\varphi]$ is a linear combination of $I,J$ and $K,$ since $\overline{u}=-u.$ Thus $H(e_2\cdot e_3\cdot\varphi,\nu\cdot\nu'\cdot\varphi)=ib\in i\R,$ and, since $\langle.,.\rangle=\Re e\ H,$ we obtain the second identity of the lemma.
\end{proof}
The lemma implies that the second term in (\ref{equation lemme B eta}) vanishes, and that (\ref{equation lemme B eta}) reduces to $\eta_2\nu_2=-\frac{1}{2}B(e_2,X);$ similarly, $\eta_3\nu_3=-\frac{1}{2}B(e_3,X).$ Using these formulas in (\ref{ecriture a priori eta}) we obtain (\ref{relation eta B}).
\section{Weierstrass representation}\label{sec representation}
We explain here how the immersion in Corollary \ref{corollary first integration} may be directly written in terms of the spinor field solution of the Dirac equation. Here, we will not use the fundamental theorem of the theory of surfaces in $\R^{1,3},$ but we will instead obtain a spinorial proof of this theorem. Note that a spinorial representation formula of a conformal immersion of a surface in $\R^{1,3}$ first appeared in \cite{Va} (see Remark \ref{rmk varlamov} below).

Assume that we have a spinor field $\varphi\in\Gamma(\Sigma)$ such that
$$D\varphi=\vec{H}\cdot\varphi\hspace{1cm}\mbox{ and }\hspace{1cm}H(\varphi,\varphi)=1.$$ 
We define the $\HC$-valued $1$-form $\xi$ by
\begin{equation}\label{def xi}
\xi(X):=\langle\langle X\cdot\varphi,\varphi\rangle\rangle\hspace{1cm}\in\hspace{.5cm}\HC,
\end{equation}
where the pairing $\langle\langle.,.\rangle\rangle:\Sigma\times\Sigma\rightarrow\HC$ is defined in (\ref{def scalar HC}). 
\begin{prop}\label{prop fundamental xi}
1- The form $\xi\in\Omega^1(M,\HC)$ satisfies
$$\xi=-\tch{\overline{\xi}},$$
and thus takes its values in $\R^{1,3}\subset\HC.$
\\2- The form $\xi$ is closed:
$$d\xi=0.$$
\end{prop}
\noindent
{\it Proof:} 
1- Using the properties (\ref{properties pairing}), we readily get
\begin{eqnarray*}
\xi(X)=\langle\langle X\cdot\varphi,\varphi\rangle\rangle=-\bigtch{\langle\langle\varphi,X\cdot\varphi\rangle\rangle}=-\bigtch{\overline{\langle\langle X\cdot\varphi,\varphi\rangle\rangle}}=-\bigtch{\overline{\xi(X)}}.
\end{eqnarray*}
2- By a straightforward computation, we get
$$d\xi(e_2,e_3)=\langle\langle e_3\cdot\nabla_{e_2}\varphi,\varphi\rangle\rangle-\langle\langle e_2\cdot\nabla_{e_3}\varphi,\varphi\rangle\rangle+\langle\langle e_3\cdot\varphi,\nabla_{e_2}\varphi\rangle\rangle-\langle\langle e_2\cdot\varphi,\nabla_{e_3}\varphi\rangle\rangle.$$
First observe that the last two terms are linked to the first two terms by
$$\langle\langle e_3\ \cdot\ \varphi,\nabla_{e_2}\varphi\rangle\rangle=-\bigtch{\overline{\langle\langle e_3\cdot\nabla_{e_2}\varphi,\varphi\rangle\rangle}}\hspace{.8cm}\mbox{and}\hspace{.8cm}\langle\langle e_2\ \cdot\ \varphi,\nabla_{e_3}\varphi\rangle\rangle=-\bigtch{\overline{\langle\langle e_2\cdot\nabla_{e_3}\varphi,\varphi\rangle\rangle}}.$$
Moreover
\begin{eqnarray*}
\langle\langle e_3\cdot\nabla_{e_2}\varphi,\varphi\rangle\rangle-\langle\langle e_2\cdot\nabla_{e_3}\varphi,\varphi\rangle\rangle&=&\bigtch{\langle\langle e_2\cdot e_3\cdot\nabla_{e_2}\varphi,e_2\cdot\varphi\rangle\rangle}-\bigtch{\langle\langle e_3\cdot e_2\cdot\nabla_{e_3}\varphi,e_3\cdot\varphi\rangle\rangle}\\
&=&\langle\langle e_2\cdot\nabla_{e_2}\varphi,e_3\cdot e_2\cdot\varphi\rangle\rangle-\langle\langle e_3\cdot\nabla_{e_3}\varphi,e_2\cdot e_3\cdot\varphi\rangle\rangle\\
&=&-\langle\langle D\varphi,e_2\cdot e_3\cdot\varphi\rangle\rangle\\
&=&-\langle\langle\vec{H} \cdot\varphi, e_2\cdot e_3\cdot\varphi\rangle\rangle.
\end{eqnarray*}
Thus
$$d\xi(e_2,e_3)=-\left(\langle\langle \vec{H} \cdot\varphi, e_2\cdot e_3\cdot\varphi\rangle\rangle-\bigtch{\overline{\langle\langle\vec{H}\cdot\varphi, e_2\cdot e_3\cdot\varphi\rangle\rangle}}\right).$$
But
\begin{eqnarray*}
\langle\langle\vec{H} \cdot\varphi, e_2\cdot e_3\cdot\varphi\rangle\rangle&=&-\bigtch{\langle\langle\varphi, \vec{H}\cdot e_2\cdot e_3\cdot\varphi\rangle\rangle}=-\bigtch{\langle\langle\varphi, e_2\cdot e_3\cdot\vec{H}\cdot\varphi\rangle\rangle}\\
&=&\bigtch{\langle\langle e_2\cdot e_3\cdot\varphi, \vec{H}\cdot\varphi\rangle\rangle}=\bigtch{\overline{\langle\langle\vec{H}\cdot\varphi, e_2\cdot e_3\cdot\varphi\rangle\rangle}}.
\end{eqnarray*}
Thus $d\xi=0.$
$\finpreuve$\\
Assume that $M$ is simply connected; then, there exists a function 
$$F:M\lgra\R^{1,3}$$ 
so that $dF=\xi$. The next theorem follows from the properties of the Clifford action:
\begin{thm}\label{theorem second integration}
\begin{enumerate}
\item The map $F=(F_0,F_1,F_2,F_3):M\lgra\R^{1,3}$ is an isometry.
\item The map 
\function{\Phi_E}{E}{M\times\R^{1,3}}{X\in E_m}{(F(m),\xi_0(X),\xi_1(X),\xi_2(X),\xi_3(X))}
is an isometry between $E$ and the normal bundle $N(F(M))$ of $F(M)$ into $\R^{1,3}$, preserving connections and second fundamental forms.
\end{enumerate}
\end{thm}
\begin{proof}
Recall that the Minkowski space $\R^{1,3}$ is identified to the subspace $\{\xi\in\HC:\ \xi=-\tch{\overline{\xi}}\},$ with the metric given by the restriction of the quadratic form $H.$ The Minkowski norm is thus 
$$\langle\xi,\xi\rangle:=H(\xi,\xi)=\langle\langle\xi,\xi\rangle\rangle\hspace{.5cm}\in \hspace{.5cm}\R.$$
We first compute, for all $X,Y$ belonging to $E\cup TM,$
$$\overline{\xi(Y)}\xi(X)=\overline{\langle\langle Y\cdot\varphi,\varphi\rangle\rangle}\langle\langle X\cdot\varphi,\varphi\rangle\rangle=\overline{\left(\overline{[\varphi]}[Y\cdot\varphi]\right)}\left(\overline{[\varphi]}[X\cdot\varphi]\right)=\overline{[Y\cdot\varphi]}[X\cdot\varphi]$$
since $[\varphi]\overline{[\varphi]}=1$ ($H(\varphi,\varphi)=1$). Here and below the brackets $[.]$ stand for the components  ($\in \HC$) of the spinor fields in some local section $\tilde{s}$ of $\tilde{Q}.$ Thus
\begin{equation}\label{formula xiX xiY}
\overline{\xi(Y)}\xi(X)=\langle\langle X\cdot\varphi,Y\cdot\varphi\rangle\rangle,
\end{equation}
which in particular implies
\begin{equation}\label{xi isom1}
\langle\xi(X),\xi(Y)\rangle=\langle X\cdot\varphi,Y\cdot\varphi\rangle.
\end{equation}
This last identity easily gives
\begin{equation}\label{xi isom2}
\langle\xi(X),\xi(Y)\rangle=0\hspace{1cm}\mbox{and}\hspace{1cm}|\xi(Z)|^2=|Z|^2
\end{equation}
for all $X\in TM,$ $Y\in E$ and $Z\in E\cup TM.$ Thus $F=\int\xi$ is an isometry, and $\xi$ maps isometrically the bundle $E$ to the normal bundle of $F(M)$ in $\R^{1,3}.$ 

We now prove that $\xi$ preserves the normal connection and the second fundamental form: let $X\in TM$ and $Y\in\Gamma(E)\cup\Gamma(TM);$ then $\xi(Y)$ is a vector field normal or tangent to $F(M).$ Considering $\xi(Y)$ as a map $M\rightarrow\R^{1,3}\subset \HC,$ we have
\begin{eqnarray}
d(\xi(Y))(X)&=&d\langle\langle Y\cdot\varphi,\varphi\rangle\rangle(X)\nonumber\\
&=&\langle\langle \nabla_XY\cdot\varphi,\varphi\rangle\rangle+\langle\langle Y\cdot\nabla_X\varphi,\varphi\rangle\rangle+\langle\langle Y\cdot\varphi,\nabla_X\varphi\rangle\rangle\label{formula dxi}
\end{eqnarray} 
where the connection $\nabla_XY$ denotes the connection on $E$ (if $Y\in\Gamma(E)$) or the Levi-Civita connection on $TM$ (if $Y\in\Gamma(TM)$). We will need the following formulas:
\begin{lem}
We have
$$\langle\ \langle\langle\nabla_XY\cdot\varphi,\varphi\rangle\rangle,\xi(\nu)\rangle=\langle\nabla_XY\cdot\varphi,\nu\cdot\varphi\rangle,$$
$$H(\langle\langle Y\cdot\nabla_X\varphi,\varphi\rangle\rangle,\xi(\nu))=H\left(Y\cdot\nabla_X\varphi,\nu\cdot\varphi\right)$$
and
$$H(\langle\langle Y\cdot\varphi,\nabla_X\varphi\rangle\rangle,\xi(\nu))=\overline{H\left(Y\cdot\nabla_X\varphi,\nu\cdot\varphi\right)}.
$$
In the expressions above, $H$ and $\langle.,.\rangle=\Re e H$ stand respectively for the complex and the real scalar products, defined on $\HC$ for the left-hand side and on $\Sigma$ for the right-hand side of each identity.
\end{lem}
\begin{proof}
The first identity is a consequence of (\ref{formula xiX xiY}) and the second identity may be obtained by a very similar computation. For the last identity, we first notice that
\begin{eqnarray*}
\langle\langle Y\cdot\varphi,\nabla_X\varphi\rangle\rangle=-\bigtch{\langle\langle \varphi,Y\cdot\nabla_X\varphi\rangle\rangle}=-\bigtch{\overline{\langle\langle Y\cdot\nabla_X\varphi,\varphi\rangle\rangle}}.
\end{eqnarray*}
Thus, using moreover that $\xi(\nu)=-\bigtch{\overline{\xi(\nu)}}$ (Proposition \ref{prop fundamental xi}), we get
\begin{eqnarray*}
H\left(\langle\langle Y\cdot\varphi,\nabla_X\varphi\rangle\rangle,\xi(\nu)\right)&=&H\left(-\bigtch{\overline{\langle\langle Y\cdot\nabla_X\varphi,\varphi\rangle\rangle}},-\bigtch{\overline{\xi(\nu)}}\right)\\
&=&H\left(\bigtch{\langle\langle Y\cdot\nabla_X\varphi,\varphi\rangle\rangle},\bigtch{\xi(\nu)}\right)\\
&=&\overline{H\left(\langle\langle Y\cdot\nabla_X\varphi,\varphi\rangle\rangle,\xi(\nu)\right)},
\end{eqnarray*}
and the result follows by the second identity of the lemma.
\end{proof}
\noindent From (\ref{formula dxi}) and the lemma, we readily get the formula 
\begin{equation}\label{formula dxi 2}
\langle d(\xi(Y))(X),\xi(\nu)\rangle=\langle\nabla_XY\cdot\varphi,\nu\cdot\varphi\rangle+2\langle Y\cdot\nabla_X\varphi,\nu\cdot\varphi\rangle.
\end{equation}
We first suppose that $X,Y\in\Gamma(TM).$ Recalling (\ref{B function phi}), and since the first term in the right-hand side of the equation above vanishes in that case ($\nabla_XY\in\Gamma(TM),$ $\nu\in\Gamma(E)$), we get
\begin{eqnarray*}
\langle d(\xi(Y))(X),\xi(\nu)\rangle&=&2\langle Y\cdot\nabla_X\varphi,\nu\cdot\varphi\rangle=\langle B(X,Y),\nu\rangle=\langle \xi(B(X,Y)),\xi(\nu)\rangle,
\end{eqnarray*}
that is, the component of $d(\xi(Y))(X)$ normal to $F(M)$ is given by
\begin{equation}\label{xi preserves B}
(d(\xi(Y))(X))^N=\xi(B(X,Y)).
\end{equation}
We now suppose that $X\in\Gamma(TM)$ and $Y\in\Gamma(E).$ We first observe that the second term in the right-hand side of equation (\ref{formula dxi 2}) vanishes: indeed, if $(e_0,e_1)$ stands for an orthonormal basis of $E,$ for all  $i,j\in\{0,1\}$ we have
\begin{eqnarray*}
\langle e_i\cdot\nabla_X\varphi,e_j\cdot\varphi\rangle=-\langle \nabla_X\varphi,e_i\cdot e_j\cdot\varphi\rangle=-\langle \eta(X)\cdot\varphi,e_i\cdot e_j\cdot\varphi\rangle,
\end{eqnarray*}
which is a sum of terms of the form $\langle e\cdot\varphi,e'\cdot\varphi\rangle$ with $e$ and $e'$ belonging to $TM$ and $E$ respectively; these terms are thus all equal to zero. Thus, (\ref{formula dxi 2}) simplifies to
$$\langle d(\xi(Y))(X),\xi(\nu)\rangle=\langle\nabla_XY\cdot\varphi,\nu\cdot\varphi\rangle=\langle\xi(\nabla_XY),\xi(\nu)\rangle$$
in that case, and thus
\begin{equation}\label{xi preserves nabla}
(d(\xi(Y))(X))^N=\xi(\nabla_XY).
\end{equation}
Equations (\ref{xi preserves B}) and (\ref{xi preserves nabla}) mean that $\Phi_E=\xi$ preserves the second fundamental form and the normal connection respectively.
\end{proof}
\begin{rem}
The immersion $F:M\rightarrow\R^{1,3}$ given by the fundamental theorem is thus
$$F=\int\xi=\left(\int\ \xi_0,\int\ \xi_1,\int\ \xi_2,\int\ \xi_3\right).$$
This formula generalizes the classical Weierstrass representation: let $\alpha_0,\alpha_1,\alpha_2,\alpha_3$ be the $\C-$linear forms defined by
$$\alpha_k(X)=\xi_k(X)-i\xi_k(JX),$$
for $k=0,1,2,3,$ where $J$ is the natural complex structure of $M.$ Let $z$ be a conformal parameter of $M,$ and let $\psi_0,\psi_1,\psi_2,\psi_3:M\rightarrow\C$ be such that
$$\alpha_0=\psi_0dz,\ \alpha_1=\psi_1dz,\ \alpha_2=\psi_2dz,\ \alpha_3=\psi_3 dz.$$
By an easy computation using $D\varphi=\vec H\cdot\varphi,$ we see that $\alpha_0,\alpha_1, \alpha_2$ and $\alpha_3$ are holomorphic forms if and only if $M$ is a maximal surface (i.e. $\vec H=\vec 0$). Thus, if $M$ is maximal, 
\begin{eqnarray*}
F&=&\left(i\ \Im m\int\alpha_0,\ \Re e\int\alpha_1,\ \Re e\int\alpha_2,\ \Re e\int\alpha_3\right)\\
&=&\left(i\ \Im m\int\psi_0dz,\ \Re e\int\psi_1dz,\ \Re e\int\psi_2dz,\ \Re e\int\psi_3dz\right)
\end{eqnarray*}
where $\psi_0,\psi_1,\psi_2,\psi_3$ are holomorphic functions. This is the Weierstrass representation of maximal surfaces in $\R^{1,3}.$
\end{rem}
\begin{rem}\label{rmk varlamov}
A representation using spinors of a conformal immersion of a surface in $\R^{1,3}$ already appeared in \cite{Va}, formula (46); the form of the representation given there is different to the representation $F=\int\xi$ given above (in particular, the normal bundle and the Clifford action do not explicitly appear in the formula given in \cite{Va}), and we don't know if one of these representations may be easily deduced from the other or not.
\end{rem}
\begin{rem}\label{rem immersed surface} 
If $M$ is a surface in $\R^{1,3},$ the immersion may be obtained from the constant spinor fields $\textit{1}$ or $-\textit{1}\in\HC$ restricted to the surface: indeed, for one of these spinor fields $\varphi,$ and for all $X\in TM\subset M\times\R^{1,3},$ we have
$$\xi(X)=\langle\langle X\cdot\varphi,\varphi\rangle\rangle=\overline{[\varphi]}[X]\hat{[\varphi]}=[X],$$
where here the brackets $[X]$ and $[\varphi]=\pm \textit{1}\in\HC$ represent $X$ and $\varphi$ in one of the two spinorial frames of $\R^{1,3}$ which are above the canonical basis. Identifying $[X]\in\R^{1,3}\subset\HC$ to $X\in\R^{1,3},$ $F=\int\xi$ identifies to the identity.
\end{rem}
As in the euclidean case \cite{BLR}, Theorem \ref{theorem second integration} gives a spinorial proof of the fundamental theorem:
\begin{cor}\label{spinorial proof fundamental theorem}
We may integrate the Gauss, Ricci and Codazzi equations in two steps:
\\
\\1- first solving
\begin{equation}\label{eqn killing 2}
\nabla_X\varphi=\eta(X)\cdot\varphi,
\end{equation}
where 
$$\eta(X)=-\frac{1}{2}\sum_{j=2,3}e_j\cdot B(e_j,X)$$
(there is a solution $\varphi$ in $\Gamma(\Sigma)$ such that $H(\varphi,\varphi)=1,$ unique up to the natural right-action of $Spin(1,3)$ on $\Gamma(\Sigma)$);
\\
\\2- then solving
$$dF=\xi$$
where $\xi(X)=\langle\langle X\cdot\varphi,\varphi\rangle\rangle$ (the solution is unique, up to translations of $\R^{1,3}\subset\HC$). 
\end{cor}
This proves the fundamental theorem, Corollary \ref{corollary1 theorem}.
\begin{proof}
Equation (\ref{eqn killing 2}) is solvable, since the Gauss, Ricci and Codazzi equations are exactly its conditions of integrability; see the proof of Theorem \ref{thm1}. Moreover, the multiplication on the right by a constant belonging to $Spin(1,3)$ in the first step, and the addition  of a constant belonging to $\R^{1,3}$ in the second step, correspond to a rigid motion in $\R^{1,3}.$
\end{proof}
\section{Surfaces in $\R^3,$ $\R^{1,2}$ and $\mathbb{H}^3.$}\label{sec hypersurfaces}
The aim of this section is to recover the spinorial characterizations of the spacelike immersions in $\R^3$ (see \cite{Fr} for a complete description), in $\R^{1,2}$ \cite{LR} and in $\mathbb{H}^3$ \cite{Mo}. We suppose that $E=\R e_0\oplus\R e_1$ where $e_0$ and $e_1$ are unit, orthogonal and parallel sections of $E$ such that $\langle e_0,e_0\rangle=-1$ and $\langle e_1,e_1\rangle=+1;$ we moreover assume that $e_0$ is future-directed, and that $(e_0,e_1)$ is positively oriented . We consider the isometric embeddings of $\R^3,$ $\R^{1,2}$ and $\mathbb{H}^3$ in $\R^{1,3}\subset\HC$ given by
$$\R^3={e_0^0}^{\perp},\  \R^{1,2}={e_1^0}^{\perp}\hspace{.2cm}\mbox{and}\hspace{.2cm}\mathbb{H}^3=\{X\in\R^{1,3}:\ \langle X,X\rangle=-1,\ X_0>0\},$$
where $e_0^0=i\1$ and $e_1^0=I$ are the first two vectors of the canonical basis of $\R^{1,3}\subset\HC.$ Let $\vec{H}$ be a section of $E$ and $\varphi\in\Gamma(\Sigma)$ be a  solution of
\begin{equation}\label{systeme phi}
D\varphi=\vec{H}\cdot\varphi\hspace{1cm}\mbox{and}\hspace{1cm}H(\varphi,\varphi)=1.
\end{equation}
According to Section \ref{sec representation}, the spinor field $\varphi$ defines an isometric immersion $M\stackrel{\varphi}{\hookrightarrow}\R^{1,3}$ (unique, up to translations),  with normal bundle $E$ and mean curvature vector $\vec{H}.$

We now give a characterization of the isometric immersions in $\R^3,$ $\R^{1,2}$ and $\mathbb{H}^3$ (up to translations) in terms of $\varphi:$
\begin{prop}\label{prop cond hypersurfaces}
1- Assume that 
\begin{equation}\label{cond phi reel}
\vec{H}=H e_1\hspace{1cm}\mbox{and}\hspace{1cm}e_1\cdot\varphi=e_2\cdot e_3\cdot\varphi
\end{equation}
where $(e_2,e_3)$ is a positively oriented orthonormal frame of $M.$ Then the isometric immersion $M\stackrel{\varphi}{\hookrightarrow}\R^{1,3}$ belongs to $\R^{3}.$
\\
\\2- Assume that 
\begin{equation}\label{condition R1,2}  
\vec{H}=H e_0\hspace{1cm}\mbox{and}\hspace{1cm}e_1\cdot\varphi=\pm i\varphi.
\end{equation}
Then the isometric immersion $M\stackrel{\varphi}{\hookrightarrow}\R^{1,3}$ belongs to $\R^{1,2}.$
\\
\\3- Consider the function $F:=\langle\langle e_0\cdot\varphi,\varphi\rangle\rangle,$ and assume that 
\begin{equation}\label{condition H3}
\vec{H}=e_0+H e_1\hspace{1cm}\mbox{and}\hspace{1cm}dF(X)=\langle\langle X\cdot\varphi,\varphi\rangle\rangle
\end{equation} 
for all $X\in\Gamma(TM).$ Then the isometric immersion $M\stackrel{\varphi}{\hookrightarrow}\R^{1,3}$ belongs to $\HH^3$
\\

Reciprocally, if $M\stackrel{\varphi}{\hookrightarrow}\R^{1,3}$ belongs to $\R^3$ (resp. to $\R^{1,2},$ $\HH^3$), then (\ref{cond phi reel}) (resp. (\ref{condition R1,2}), (\ref{condition H3})) holds for some unit, orthogonal and parallel sections $(e_0,e_1)$ of $E.$ 
\end{prop}
\noindent\textit{Proof:} 1- We suppose that (\ref{cond phi reel}) holds, and we compute
\begin{eqnarray*}
\xi(e_0)&=&\langle\langle e_0\cdot\varphi,\varphi\rangle\rangle=-\bigtch{\langle\langle e_0\cdot e_1\cdot\varphi,e_1\cdot\varphi\rangle\rangle}=-\bigtch{\langle\langle e_0\cdot e_1\cdot\varphi,e_2\cdot e_3\cdot \varphi\rangle\rangle}\\&=&\bigtch{\langle\langle e_0\cdot e_1\cdot e_2\cdot e_3\cdot\varphi,\varphi\rangle\rangle}=i\1
\end{eqnarray*}
since, by a direct computation, $[e_0\cdot e_1\cdot e_2\cdot e_3\cdot\varphi]=-i[\varphi],$ where the brackets denote the complex quaternions which represent the spinor fields in a given frame $\tilde{s}\in \tilde{Q}$ adapted to the basis $(e_0,e_1,e_2,e_3)$ (see Section \ref{notation generale}). Thus, the constant vector $e_0^0=i\1\in \R^{1,3}\subset \HC$ is normal to the immersion (by Theorem \ref{theorem second integration}, (2)), and the immersion thus belongs to $\R^3.$ 
\\2- Analogously, assuming that (\ref{condition R1,2}) holds, we compute
\begin{equation*}
\xi(e_1)=\langle\langle e_1\cdot\varphi,\varphi\rangle\rangle=\pm\langle\langle i\varphi,\varphi\rangle\rangle=\pm\overline{[\varphi]}([\varphi] I)=\pm I
\end{equation*}
where $[\varphi]\in\HC$ represents $\varphi$ in $\tilde{s}.$ The constant vector $e_1^0=I$ is thus normal to the immersion, and the result follows. 
\\3- Assuming that (\ref{condition H3}) holds, the function $F=\langle\langle e_0\cdot\varphi,\varphi\rangle\rangle$ is a primitive of the 1-form $\xi(X)=\langle\langle X\cdot\varphi,\varphi\rangle\rangle,$ and is thus the isometric immersion defined by $\varphi$ (uniquely defined, up to translations); since the Minkowski norm of $\langle\langle e_0\cdot\varphi,\varphi\rangle\rangle\in\R^{1,3}\subset\HC$ coincides with the norm of $e_0,$ and is thus constant equal to $-1,$ and since, by a direct computation, its first component is positive, the immersion belongs to $\HH^3.$ 

For the converse statements, we choose  $(e_0,e_1)$ such that $\xi(e_0)=\langle\langle e_0\cdot\varphi,\varphi\rangle\rangle=i\1$ in the first case (the other condition $\langle\langle e_0\cdot\varphi,\varphi\rangle\rangle=-i\1$, expressing that the normal vector $\xi(e_0)$ is the other unit vector $-i\1$ normal to $\R^3,$ is not compatible with $H(\varphi,\varphi)=1$), such that $\xi(e_1)=\langle\langle e_1\cdot\varphi,\varphi\rangle\rangle=\pm I$ in the second case, and such that $\xi(e_0)=\langle\langle e_0\cdot\varphi,\varphi\rangle\rangle$ is the future-directed vector normal to $\HH^3$ in $\R^{1,3}$ in the last case. Writing, for the first two cases, the spinors in a frame $\tilde{s}$ adapted to $(e_0,e_1,e_2,e_3)$, we easily deduce (\ref{cond phi reel}) and (\ref{condition R1,2}). For the last case, (\ref{condition H3}) is immediate since $\langle\langle e_0\cdot\varphi,\varphi\rangle\rangle$ is the immersion.
$\finpreuve$

We now indicate how these results may be related to the criteria concerning immersions in $\R^3,$ $\R^{1,2}$ and $\mathbb{H}^3$; for briefness, we omit most of the proofs.

Assume first that  $M\subset \mathcal{H}\subset\R^{1,3},$ where $\mathcal{H}$ is the spacelike hyperplane $\R^3$, or the unit hyperboloid $\mathbb{H}^3$ of $\R^{1,3},$ and consider $e_0$ and $e_1$ timelike and spacelike unit vector fields such that
 $$\R^{1,3}=\R e_0\oplus_{\perp} T\mathcal{H}\hspace{1cm}\mbox{and}\hspace{1cm}T\mathcal{H}=\R e_1\oplus_{_\perp}TM.$$
 The intrinsic spinors of $M$ identify with the spinors of $\mathcal{H}$ restricted to $M,$ which in turn identify with the positive spinors of $\R^{1,3}$ restricted to $M:$
 \begin{prop}
There is an identification
  \begin{eqnarray*}
  \Sigma M &\stackrel{\sim}{\rightarrow}& \Sigma^+_{|M}\\
  \psi&\mapsto&\psi^*
  \end{eqnarray*}
such that, for all $X\in TM$ and all $\psi\in \Sigma M,$
$\displaystyle{(\nabla_X\psi)^*=\nabla_X(\psi^*),}$
the Clifford actions are linked by
$$(X\cdot_{_M}\psi)^*=X\cdot e_1\cdot \psi^*,$$
and the following two properties hold:
\begin{equation}\label{ident cond1}
H(e_1\cdot\psi^*,e_2\cdot e_3\cdot\psi^*)=\frac{1}{2}|\psi|^2
\end{equation}
and 
\begin{equation}\label{ident cond2} 
d\langle\langle e_0\cdot\psi^*,\psi^*\rangle\rangle(X)=\langle\langle X\cdot\psi^*,\psi^*\rangle\rangle\hspace{.5cm}\mbox{ iff }\hspace{.5cm}d|\psi|^2(X)=-\Re e\langle X\cdot_{_M}\overline{\psi},\psi\rangle.
\end{equation}
 \end{prop}
Using this identification, the intrinsic Dirac operator on $M$ defined by
$$D_M\psi:=e_2\cdot_{_M}\nabla_{e_2}\psi+e_3\cdot_{_M}\nabla_{e_3}\psi$$
 is linked to $D$ by
 $$(D_M\psi)^*=-e_1\cdot D\psi^*.$$
If $\varphi\in\Gamma(\Sigma)$ is a solution of
 \begin{equation}\label{systeme phi 2}
 D\varphi=\vec H\cdot\varphi\hspace{1cm}\mbox{ and }\hspace{1cm}H(\varphi,\varphi)=1
 \end{equation}
we may consider $\psi\in\Sigma M$ such that $\psi^*=\varphi^+$; it satisfies 
\begin{equation}\label{eqnphi+}
(D_M\psi)^* =-e_1\cdot D\psi^*=-e_1\cdot \vec H\cdot \psi^*.
\end{equation}
Note that $\psi\ne 0,$ since 
\begin{equation}\label{norme phi phi+ phi-}
H(\varphi,\varphi)=2H(\varphi^+,\varphi^-)=1
\end{equation}
(here and below the decomposition $\varphi=\varphi^++\varphi^-$ is the decomposition in \\$\Sigma=\Sigma^+\oplus\Sigma^-,$ and we recall that $H$ vanishes on $\Sigma^+$ and $\Sigma^-;$ see Section \ref{twisted spinor bundle}).
\\
\\We examine separately the case of a surface in $\R^3,$ and in $\mathbb{H}^3:$
\\
\\1. If $\mathcal{H}=\R^3$, then $\vec H$ is of the form $He_1,$ and (\ref{eqnphi+}) reads
\begin{equation*}\label{eqn Friedrich}
D_M\psi =H\psi;
\end{equation*}
moreover, (\ref{norme phi phi+ phi-}), (\ref{cond phi reel}) and (\ref{ident cond1}) imply that $|\psi|=1.$ This is the spinorial characterization of an isometric immersion in $\R^3.$
\\
\\2. If $\mathcal{H}=\mathbb{H}^3,$ then $\vec H$ is of the form $e_0+He_1,$ 
and  (\ref{eqnphi+}) reads
\begin{equation}\label{eqn Morel 1}
D_M\psi =H\psi+\overline{\psi};
\end{equation}
moreover, it is not difficult to prove that (\ref{condition H3}) implies that (\ref{ident cond2}) holds. Equation (\ref{eqn Morel 1}), together with the right hand side of (\ref{ident cond2}), is the spinorial characterization of an isometric immersion in $\mathbb{H}^3$ obtained by B. Morel in \cite{Mo}.
\\

Assume now that $M\subset \mathcal{H}\subset\R^{1,3},$ where $\mathcal{H}$ is the timelike hyperplane $\R^{1,2}$, and let $e_0$ and $e_1$ be timelike and spacelike unit vector fields such that
 $$\R^{1,3}=\R e_1\oplus_{\perp} T\mathcal{H}\hspace{1cm}\mbox{and}\hspace{1cm}T\mathcal{H}=\R e_0\oplus_{_\perp}TM.$$
\begin{prop}
There is an identification
  \begin{eqnarray*}
  \Sigma M &\stackrel{\sim}{\rightarrow}& \Sigma^+_{|M}\\
  \psi&\mapsto&\psi^*
  \end{eqnarray*}
such that, for all $X\in TM$ and all $\psi\in \Sigma M,$ $\displaystyle{(\nabla_X\psi)^*=\nabla_X(\psi^*),}$ the Clifford actions are linked by
$$(X\cdot_{_M}\psi)^*=ie_0\cdot X\cdot \psi^*$$
and 
$$H(\psi^*,ie_1\cdot\psi^*)=\frac{1}{2}\left(|\psi^+|^2-|\psi^-|^2\right).$$
\end{prop}
Setting $\psi\in \Sigma M$ such that $\psi^*=\varphi^+,$ and using (\ref{condition R1,2}), the equation $D\varphi=\vec H\cdot\varphi$ with $H(\varphi,\varphi)=1$ and $\vec{H}=H e_0$ reads 
\begin{equation}\label{spineur R1,2}
D_M\psi=iH\psi\hspace{1cm}\mbox{with}\hspace{1cm}|\psi^+|^2-|\psi^-|^2=\pm 1.
\end{equation}
Reciprocally, to a solution $\psi$ of (\ref{spineur R1,2}) corresponds a solution $\varphi$ of (\ref{systeme phi 2}), defined by $\varphi^+=\psi^*$ and $\varphi^-=\pm i e_1\cdot\psi^*.$ A solution of (\ref{spineur R1,2}) is thus equivalent to an isometric immersion of the surface in  $\R^{1,2}.$ We thus obtain a spinorial characterization of the isometric immersions of a Riemmanian surface in $\R^{1,2},$ which is different to the characterization obtained in \cite{LR}, where two spinor fields are needed.
\begin{rem}\label{rmq R1,2}
We also may obtain an explicit representation formula: let us consider the indefinite hermitian inner product on $\Sigma M$ defined by
$$\langle\psi,\psi'\rangle:=\langle\psi^+,\psi'^+\rangle-\langle\psi^-,\psi'^-\rangle$$
for all $\psi,\psi'\in\Sigma M,$ where the hermitian products on $\Sigma^+M$ and $\Sigma^-M$ are the usual inner products. If $\varphi=(1\pm i e_1)\cdot\psi^*$ where $\psi$ satisfies (\ref{spineur R1,2}),  it may be proved by a computation that 
\begin{eqnarray}\label{representation R1,2}
\hspace{-.5cm}\xi(X)&=&\langle\langle X\cdot\varphi,\varphi\rangle\rangle\nonumber\\
&=&\hspace{-.3cm}-\langle X\cdot_{_M}\psi,\overline{\psi}\rangle\pm J\left(\Re e\langle X\cdot_{_M}\psi,\alpha(\psi)\rangle +\Im m\langle X\cdot_{_M}\psi,\alpha(\psi)\rangle I\right)
\end{eqnarray}
where $\overline{\psi}=\psi^+-\psi^-$ and $\alpha:\Sigma M\rightarrow\Sigma M$ is a quaternionic structure. Observe that $|\overline{\psi}|^2=-|\alpha(\psi)|^2=\pm 1$ and $\langle\overline{\psi},\alpha(\psi)\rangle=0,$ so that $\xi(X)$ may be interpreted as the coordinates $\in i\R\oplus\C$ of $X\cdot_{_M}\psi$ in the orthonormal basis $(\overline{\psi},\alpha(\psi)).$ According to Section \ref{sec representation}, the immersion is given by $F=\int\xi.$ This formula is similar to the representation formula for surfaces in $\R^3$ obtained by Friedrich in \cite{Fr}. For sake of briefness we omitted the proof of (\ref{representation R1,2}); a direct proof may also be easily obtained, showing by a computation that formula (\ref{representation R1,2}) defines an isometric immersion in $\R^{1,2}\subset\HC.$
\end{rem}
\section{Weierstrass representation of flat surfaces with flat normal bundle in $\R^{1,3}$}\label{sec flat}
We study here the flat surfaces with flat normal bundle and regular Gauss map in $\R^{1,3}$ using spinors, and give an alternative proof of the main result in \cite{GMM}.
\subsection{The Gauss map of a spacelike surface in $\R^{1,3}$}\label{ss gauss map} 
We consider $\Lambda^2\R^{1,3},$ the vector space of bivectors of $\R^{1,3}$ endowed with its natural metric $\langle.,.\rangle$ (which has signature (3,3)). The Grassmannian of the oriented spacelike 2-planes in $\R^{1,3}$ identifies with the submanifold of unit and simple bivectors
$$\mathcal{Q}=\{\eta\in\Lambda^2\R^{1,3}:\ \langle \eta,\eta\rangle=1,\ \eta\wedge\eta=0\},$$
and the oriented Gauss map with the map
$$G:\ M\rightarrow \mathcal{Q},\ p\mapsto G(p)=u_1\wedge u_2,$$
where $(u_1,u_2)$ is a positively oriented orthonormal basis of $T_pM.$ The Hodge $*$ operator $\Lambda^2\R^{1,3}\rightarrow \Lambda^2\R^{1,3}$ is defined by the relation
\begin{equation}\label{i lambda2}
\langle *\eta,\eta'\rangle=\eta\wedge\eta'
\end{equation}
for all $\eta,\eta'\in \Lambda^2\R^{1,3},$ where we identify $\Lambda^4\R^{1,3}$ to $\R$ using the canonical volume element $e_0^o\wedge e_1^o\wedge e_2^o\wedge e_3^o$ on $\R^{1,3};$ here and below $(e_0^o,e_1^o,e_2^o,e_3^o)$ stands for the canonical basis of $\R^{1,3}$. It satisfies $*^2=-id_{\Lambda^2\R^{1,3}}$ and thus $i:=-*$ defines a complex structure on $\Lambda^2\R^{1,3}.$ We also define
\begin{equation}\label{H lambda2}
H(\eta,\eta')=\langle \eta,\eta'\rangle-i\ \eta\wedge\eta'\hspace{.5cm}\in\hspace{.5cm}\C
\end{equation}
for all $\eta,\eta'\in\ \Lambda^2\R^{1,3}.$ This is a $\C$-bilinear map on $\Lambda^2\R^{1,3},$ and we have
$$\mathcal{Q}=\{\eta\in\Lambda^2\R^{1,3}:\ H(\eta,\eta)=1\}.$$
The bivectors
$$E_1=e_2^o\wedge e_3^o,\ E_2=e_3^o\wedge e_1^o,\ E_3=e_1^o\wedge e_2^o$$
form a basis of $\Lambda^2\R^{1,3}$ as a vector space over $\C;$ this basis is such that $H(E_i,E_j)=\delta_{ij}$ for all $i,j.$ Identifying $\Lambda^2\R^{1,3}$ with the elements of order 2 of $Cl_0(1,3)\simeq \HC$ (see (\ref{identification Cl0})), and using the Clifford map (\ref{Clifford map}), we easily get
$$E_1=I,\ E_2=J,\ E_3=K$$
and
$$\Lambda^2\R^{1,3}=\{Z_1I+Z_2J+Z_3K\in\HC:(Z_1,Z_2,Z_3)\in\C^3\};$$
moreover, using this identification, the complex structure $i$ and the quadratic map $H$ defined above on $\Lambda^2\R^{1,3}$ coincide with the natural complex structure $i$ and the quadratic map $H$ defined on $\HC,$ and
$$\mathcal{Q}=\{Z_1I+Z_2J+Z_3K:\ Z_1^2+Z_2^2+Z_3^2=1\}\ =\ Spin(1,3)\ \cap\ \Im m\ \HC,$$
where $\Im m\ \HC$ stands for the linear space generated by $I,$ $J$ and $K$ in $\HC.$

We now suppose that the immersion of $M$ in $\R^{1,3}$ is given by some spinor field $\varphi\in\Gamma(\Sigma E\otimes\Sigma M)$ solution of $D\varphi=\vec{H}\cdot\varphi$ and such that $H(\varphi,\varphi)=1.$ We first express the Gauss map of the immersion in terms of $\varphi:$
\begin{lem}\label{gauss map function phi}
The Gauss map is given by 
\begin{eqnarray*}
G:M&\rightarrow&\mathcal{Q}\\
x&\mapsto&\langle\langle u_1\cdot u_2\cdot \varphi,\varphi\rangle\rangle
\end{eqnarray*}
where, for all $x\in M,$ $(u_1,u_2)$ is a positively oriented and orthonormal basis of $T_xM.$
\end{lem}
\begin{proof}
Setting $v_1=\langle\langle u_1\cdot\varphi,\varphi\rangle\rangle$ and $v_2=\langle\langle u_2\cdot\varphi,\varphi\rangle\rangle\in\R^{1,3}\subset\HC,$ the basis $(v_1,v_2)$ is an orthonormal basis of the immersion of $M$ in $\R^{1,3}\subset\HC$ (Theorem \ref{theorem second integration} (1)), and 
$$v_1\wedge v_2\simeq\left(\begin{array}{cc} 0&v_1\\\tch{v_1}&0\end{array}\right)\left(\begin{array}{cc} 0&v_2\\\tch{v_2}&0\end{array}\right)=\left(\begin{array}{cc} v_1\ \tch{v_2}&0\\0&\tch{v_1}\ v_2\end{array}\right)\simeq v_1\ \tch{v_2} ,$$
since the wedge product $v_1\wedge v_2\in \Lambda^2\R^{1,3}\subset Cl_0(\R^{1,3})$ of the two orthogonal vectors $v_1$ and $v_2\in\R^{1,3}\subset Cl_1(\R^{1,3})$ identifies with their Clifford product.
But
\begin{eqnarray*}
v_1\cdot\bigtch{v_2}&=&\langle\langle u_1\cdot\varphi,\varphi\rangle\rangle\bigtch{\langle\langle u_2\cdot\varphi,\varphi\rangle\rangle}=\left(\overline{[\varphi]}[u_1]\tch{[\varphi]}\right)\bigtch{\left(\overline{[\varphi]}[u_2]\tch{[\varphi]}\right)}\\
&=&\overline{[\varphi]}[u_1]\tch{[u_2]}[\varphi]=\langle\langle u_1\cdot u_2\cdot\varphi,\varphi\rangle\rangle,
\end{eqnarray*}
where $[u_1],$ $[u_2]$ and $[\varphi]\in \HC$ represent $u_1,$ $u_2$ and $\varphi$ in some frame $\tilde{s}$ of $\tilde{Q}.$
\end{proof}
The Gauss map is linked to the second fundamental form as follows:
\begin{lem}\label{lemma expr eta G} Define $\tilde{\eta}:=\langle\langle\eta\cdot\varphi,\varphi\rangle\rangle$ where $\eta$ is linked to the second fundamental form $B$ by (\ref{relation eta B}). The form $\tilde{\eta}$ satisfies
\begin{equation}\label{expr eta G}
\tilde{\eta}=\frac{1}{2}G^{-1}dG.
\end{equation}
In particular $\tilde{\eta}$ belongs to $\Omega^1(M,\mathcal{G})$ where $\mathcal{G}=\Im m\ \HC$ is the Lie algebra of $Spin(1,3).$
\end{lem}
\begin{proof}
We suppose that $(u_1,u_2)$ is a moving frame on $M$ such that $\nabla {u_i}_{|p}=0,$ and, using  $\nabla_X\varphi=\eta(X)\cdot\varphi$, we compute
\begin{eqnarray*}
dG(X)&=&\langle\langle u_1\cdot u_2\cdot\nabla_X\varphi,\varphi\rangle\rangle+\langle\langle u_1\cdot u_2\cdot\varphi,\nabla_X\varphi\rangle\rangle\\
&=&\langle\langle u_1\cdot u_2\cdot\eta(X)\cdot\varphi,\varphi\rangle\rangle+\langle\langle u_1\cdot u_2\cdot\varphi,\eta(X)\cdot\varphi\rangle\rangle\\
&=&2\langle\langle u_1\cdot u_2\cdot\eta(X)\cdot\varphi,\varphi\rangle\rangle.
\end{eqnarray*}
But 
\begin{eqnarray*}
\langle\langle u_1\cdot u_2\cdot\eta(X)\cdot\varphi,\varphi\rangle\rangle&=&\overline{[\varphi]}\ \left([u_1\cdot u_2]\ [\eta(X)]\right)\ [\varphi]\\&=&\left(\overline{[\varphi]}\ [u_1\cdot u_2]\ [\varphi]\right)\ \left(\overline{[\varphi]}\ [\eta(X)]\ [\varphi]\right)\\&=&\langle\langle u_1\cdot u_2\cdot\varphi,\varphi\rangle\rangle\langle\langle\eta(X)\cdot\varphi,\varphi\rangle\rangle,
\end{eqnarray*}
where $[\varphi],$ $[u_1\cdot u_2]$ and $[\eta(X)]\in\HC$ represent $\varphi,$ $u_1\cdot u_2$ and $\eta(X)$ respectively in some local frame $\tilde{s}$ of $\tilde{Q}.$ Thus
$$dG(X)=2G\ \tilde{\eta}(X),$$
which implies (\ref{expr eta G}).
\end{proof}

We define the \textit{vectorial product} of two vectors $\xi,\xi'\in\Im m\ \HC$ by
$$\xi \times \xi':=\frac{1}{2}\left(\xi\xi'-\xi'\xi\right)\ \in\ \Im m\ \HC.$$ 
It is such that
$$\langle\langle \xi,\xi'\rangle\rangle=H(\xi,\xi')\1+\xi \times \xi'.$$
We also define the \textit{mixed product} of three vectors $\xi,\xi',\xi''\in \Im m\ \HC$ by
$$[\xi,\xi',\xi'']:= H(\xi \times \xi',\xi'')\ \in\ \C.$$
The mixed product is a \textit{complex volume form} on $\Im m\ \HC$ (i.e. with complex values, $\C$-linear and skew-symmetric w.r.t. the three arguments); it induces a natural \textit{complex area form} $\omega_{\mathcal{Q}}$ on $\mathcal{Q}$ by
$${\omega_{\mathcal{Q}}}_p(\xi,\xi'):=[\xi,\xi',p]$$ 
for all $p\in\mathcal{Q}$ and all $\xi,\xi'\in T_p\mathcal{Q}.$ Note that ${\omega_{\mathcal{Q}}}_p(\xi,\xi')=0$ if and only if $\xi$ and $\xi'$ are dependent over $\C.$

We now compute the pull-back by the Gauss map of the area form $\omega_Q:$
\begin{prop}\label{prop pull back}
We have
\begin{equation}\label{formula pull-back omega}
G^*\omega_{\mathcal{Q}}=(K+iK_N)\omega_M,
\end{equation}
where $\omega_M$ is the form of area of $M.$ In particular, $K=K_N=0$ at $x_o\in M$ if and only if the linear space $dG_{x_o}(T_{x_o}M)$ belongs to some complex line in $T_{G(x_o)}\mathcal{Q}.$
\end{prop}
\begin{proof}
By definition,
\begin{equation}\label{G*omega1}
G^*\omega_Q(u,v)=H(dG(u)\times dG(v),G)
\end{equation}
with
\begin{equation}\label{G*omega2}
dG(u)\times dG(v)=\frac{1}{2}\left(dG(u)dG(v)-dG(v)dG(u)\right).
\end{equation}
Using successively $\tilde{\eta}=\frac{1}{2}G^{-1}dG$ (Lemma \ref{lemma expr eta G}), $G^{-1}=\overline{G}=-G,$ $GdG=-dG G$ (since $GG=-G\overline{G}=-1$) and $GG=-1,$ we easily get
\begin{equation}\label{eta tilde 2 dG}
\tilde{\eta}(u)\tilde{\eta}(v)-\tilde{\eta}(v)\tilde{\eta}(u)=\frac{1}{4}\left(dG(u)dG(v)-dG(v)dG(u)\right).
\end{equation}
We now compute the left-hand side of (\ref{eta tilde 2 dG}). Since
\begin{eqnarray*}
\langle\langle\eta(u)\cdot\varphi,\varphi\rangle\rangle\langle\langle\eta(v)\cdot\varphi,\varphi\rangle\rangle&=&\left(\overline{[\varphi]}\ [\eta(u)]\ [\varphi]\right)\ \left(\overline{[\varphi]}\ [\eta(v)]\ [\varphi]\right)\\
&=&\overline{[\varphi]}\ [\eta(u)]\ [\eta(v)]\ [\varphi]\\
&=&\langle\langle\eta(u)\cdot\eta(v)\cdot\varphi,\varphi\rangle\rangle,
\end{eqnarray*}
where $[\varphi],$ $[\eta(u)]$ and $[\eta(v)]\in\HC$ represent $\varphi,$ $\eta(u)$ and $\eta(v)$ respectively in some local frame $\tilde{s}$ of $\tilde{Q},$ we get
\begin{equation}\label{eta tilde 2 eta}
\tilde{\eta}(u)\tilde{\eta}(v)-\tilde{\eta}(v)\tilde{\eta}(u)=\langle\langle\left(\eta(u)\cdot\eta(v)-\eta(v)\cdot\eta(u)\right)\cdot\varphi,\varphi\rangle\rangle.
\end{equation}
We have seen in Lemma \ref{lem calculs courbure} (3) that
$$\eta(e_2)\cdot\eta(e_3)-\eta(e_3)\cdot\eta(e_2)=\frac{1}{2}Ke_2\cdot e_3+\frac{1}{2}K_Ne_0\cdot e_1.$$
Thus, using $\langle\langle e_2\cdot e_3\cdot\varphi,\varphi\rangle\rangle=G$ and $\langle\langle e_0\cdot e_1\cdot\varphi,\varphi\rangle\rangle=iG$ (by a direct computation, or since, for $i:=-e_0\cdot e_1\cdot e_2\cdot e_3,$ we have $e_0\cdot e_1=ie_2\cdot e_3$ and the pairing $\langle\langle.,.\rangle\rangle$ is $\C$-linear with respect to this complex structure, see Section \ref{twisted spinor bundle}), (\ref{eta tilde 2 eta}) implies
$$\tilde{\eta}(e_2)\tilde{\eta}(e_3)-\tilde{\eta}(e_3)\tilde{\eta}(e_2)=\frac{1}{2}\left(K+iK_N\right)G.$$
By (\ref{G*omega1}), (\ref{G*omega2}) and (\ref{eta tilde 2 dG}), and since $H(G,G)=1,$ the result follows.
\end{proof}
\begin{rem}\label{gauss-bonnet whitney}
Using that $d\omega_Q=0$ and the fact that the euclidean sphere 
$$S^2=\{X_1I+X_2J+X_3K:\ X_1,X_2,X_3\in\R,\ X_1^2+X_2^2+X_3^2=1\}\hspace{.5cm}\subset\hspace{.3cm} Q$$ 
is a deformation retract of $Q,$ it is easy to deduce the Gauss-Bonnet and the Whitney theorems from (\ref{formula pull-back omega}): indeed, $G:M\rightarrow Q$ is homotopically equivalent to some function $H:M\rightarrow S^2\subset Q,$ and
$$\int_M(K+iK_N)\omega_M=\int_M G^*\omega_Q=\int_M H^*\omega_Q=\int_M H^*\omega_{S^2}=4\pi.\deg H.$$
Thus, 
$$\int_MK\omega_M=4\pi.\deg H\hspace{1cm}\mbox{and} \hspace{1cm}\int_MK_N\omega_M=0.$$
\end{rem}
As a consequence of the proposition, if $K=K_N=0$ and if $G:M\rightarrow\mathcal{Q}$ is a regular map (i.e. if $dG_x$ is injective at every point $x$ of $M$), there is a unique complex structure $\mathcal{J}$ on $M$ such that
$$dG_x(\mathcal{J}u)=i.dG_x(u)$$
for all $x\in M$ and all $u\in T_xM.$ This complex structure coincides with the complex structure considered in \cite{GMM}. Note that $M$ cannot be compact under these hypotheses, since, on the Riemann surface $(M,\mathcal{J}),$ the Gauss map $G=G_1I+G_2J+G_3K$ is globally defined, non-constant, and such that  $G_1,G_2$ and $G_3$ are holomorphic functions. Thus, assuming moreover that $M$ is simply connected, by the uniformization theorem $(M,\mathcal{J})$ is conformal to an open set of $\C,$ and thus admits a globally defined conformal parameter $z=x+iy.$ 

\subsection{Spinorial description of flat immersions with flat normal bundle}\label{spinor description}

In this section we suppose that the hypotheses of Corollary \ref{corollary first integration} hold, that $M$ is simply connected and that the bundles $TM$ and $E$ are flat ($K=K_N=0).$ Recall that the bundle $\Sigma:=\Sigma E\otimes \Sigma M$ is associated to the principal bundle $\tilde{Q}$ and to the representation $\rho_1\otimes\rho_2$ of the structure group $Spin(1,1)\times Spin(2)$ in $\HC;$ see Section \ref{twisted spinor bundle}. Since the curvatures $K$ and $K_N$ are zero, the spinorial connection on the bundle $\tilde{Q}$ is flat, and $\tilde{Q}$ admits a parallel local section $\tilde{s};$ since $M$ is simply connected, the section $\tilde{s}$ is in fact globally defined. We consider $\varphi\in\Gamma(\Sigma)$ a solution of
$$D\varphi=\vec H\cdot\varphi$$
such that $H(\varphi,\varphi)=1,$ and $g:M\rightarrow Spin(1,3)\subset\HC$ such that
$$\varphi=:[\tilde{s},g]\hspace{.5cm}\in\hspace{.5cm}\tilde{Q}\times\HC/\rho_1\otimes\rho_2\hspace{.2cm}\simeq\hspace{.2cm}\Sigma E\otimes\Sigma M,$$
that is, $g\in\HC$ represents $\varphi$ in the parallel section $\tilde{s}.$ With the notation introduced Section \ref{notation generale}, we write $g=[\varphi]\in\HC.$
\subsubsection{The basic result}
\begin{thm}\label{prop etap}
Assume that the Gauss map $G$ of the immersion defined by $\varphi$ is regular, and consider the complex structure $\mathcal{J}$ on $M$ such that $G$ is an holomorphic map (see the end of Section \ref{ss gauss map}). Then, the 1-form $\eta':=dg\ g^{-1}$ represents the form $\eta$ in the frame $\tilde{s},$ and reads
\begin{equation}\label{etap theta}
\eta'=\theta_1 J+\theta_2 K,
\end{equation}
where $\theta_1$ and $\theta_2$ are two holomorphic 1-forms. 
\end{thm}
\begin{proof}
Let $(e_0,e_1)$ and $(u_1,u_2)$ be the orthonormal moving frames of $E$ and $TM$ such that $\pi(\tilde{s})=((e_0,e_1),(u_1,u_2)),$ where $\pi=\tilde{Q}\rightarrow Q_1\times_M Q_2$ is the natural projection onto the bundle $Q_1\times_M Q_2$ of the orthonormal frames of $E\oplus TM.$ Since $\tilde{s}$ is a parallel frame, so are $(e_0,e_1)$ and $(u_1,u_2).$ By Lemma \ref{lem1}, $\varphi$ satisfies
$$\nabla_X\varphi=\eta(X)\cdot\varphi\hspace{.5cm}\mbox{with}\hspace{.5cm}\eta(X)=-\frac{1}{2}\sum_{j=1,2}u_j\cdot B(u_j,X)$$
for all $X\in\Gamma(TM),$ where $B$ is given by (\ref{B function phi}). Thus $[\nabla_X\varphi]=[\eta(X)][\varphi]$ in $\tilde{s},$ which reads $dg(X)=[\eta(X)]g.$ Thus $\eta'=dg\ g^{-1}=[\eta],$ that is $\eta'$ represents $\eta$ in $\tilde{s}$ and 
$$\eta'(X)=-\frac{1}{2}\sum_{j=1,2}[u_j] \tch{[B(u_j,X)]}.$$ 
Since $[u_1]=J,$ $[u_2]=K$ and $[B(u_j,X)]$ is of the form $a(X)i\textit{1}+b(X)I,$ $a,b\in\Gamma(T^*M)$ (the vector $B(u_j,X)$ is normal to the surface), $\eta'$ is of the form (\ref{etap theta}) where $\theta_1$ and $\theta_2$ are two complex-valued 1-forms. In order to prove that $\theta_1$ and $\theta_2$ are holomorphic 1-forms we will need the following two lemmas:
\begin{lem}\label{etap2}The Gauss map of the immersion defined by $\varphi$ is given by 
\begin{equation}\label{G function g}
G=g^{-1}Ig,
\end{equation}
and the form $\tilde{\eta}=\langle\langle\eta\cdot\varphi,\varphi\rangle\rangle\in\Omega^1(M,\mathcal{G})$ by
\begin{equation}\label{expr eta g}
\tilde{\eta}=\frac{1}{2}G^{-1}dG=g^{-1}dg.
\end{equation}
\end{lem}
\noindent\textit{Proof of Lemma \ref{etap2}:} First, by Lemma \ref{gauss map function phi},
\begin{equation*}
G=\langle\langle u_1\cdot u_2\cdot\varphi,\varphi\rangle\rangle=\overline{g}Ig
\end{equation*}
since $[\varphi]=g$ and $[u_1\cdot u_2]=I$ in $\tilde{s};$ this gives (\ref{G function g}) since $\overline{g}=g^{-1}$ ($g\in Spin(1,3)$). We prove (\ref{expr eta g}): the first identity is (\ref{expr eta G}) in Lemma \ref{lemma expr eta G}. The last identity is straightforward using 
\begin{equation}\label{etatilde etap}
\tilde{\eta}=\langle\langle\eta\cdot\varphi,\varphi\rangle\rangle=g^{-1}[\eta]g=g^{-1}\eta' g
\end{equation}
with $\eta'=dg\ g^{-1}.$
$\finpreuve$
Formula (\ref{G function g}) in Lemma \ref{etap2} together with the special form of (\ref{etap theta}) may be rewritten as follows:
\begin{lem}\label{rmk lift}
Consider the projection
\begin{eqnarray*}
p:Spin(1,3)&\rightarrow &\mathcal{Q}\\
g&\mapsto& g^{-1}Ig
\end{eqnarray*}
as a $S^1_{\C}$ principal bundle, where the action of $S^1_{\C}$ on $Spin(1,3)$ is given by the multiplication on the left. This fibration is formally analogous to the classical Hopf fibration $S^3\subset\HH\rightarrow S^2\subset \Im m\ \HH,\ g\mapsto g^{-1}Ig$. It is equipped with the horizontal distribution given at every $g\in Spin(1,3)$ by
$$\mathcal{H}_g:={d(R_{g^-1})_g}^{-1}\left(\C J\oplus\C K\right)\hspace{1cm}\subset\hspace{.5cm}T_gSpin(1,3)$$
where $R_{g^-1}$ stands for the right-multiplication by $g^{-1}$ on $Spin(1,3).$ The distribution $(\mathcal{H}_g)_{g\in Spin(1,3)}$ is $H$-orthogonal to the fibers of $p,$ and, for all $g\in Spin(1,3),$ $dp_g:\mathcal{H}_g\rightarrow T_{p(g)}Q$ is an isomorphism which preserves $i$ and such that 
$$H(dp_g(u),dp_g(u))=4H(u,u)$$
for all $u\in\mathcal{H}_g.$ With these notations, we have
\begin{equation}\label{relation G g}
G=p\circ g,
\end{equation}
and the map $g:M\rightarrow Spin(1,3)$ appears to be a horizontal lift to $Spin(1,3)$ of the Gauss map $G:M\rightarrow\mathcal{Q}$ (formulas (\ref{G function g}) and (\ref{etap theta})).
\end{lem}
We achieve the proof of the theorem, showing that $g$ is a holomorphic map. From (\ref{relation G g}), we get
$$dG=dp\circ dg.$$
Since $dp$ and $dG$ commute to the complex structures $i$ defined on $Spin(1,3),$ $\mathcal{Q}$ and $M$ , so does $dg,$ and the result follows.
\end{proof}
\subsubsection{The immersion depends on two holomorphic functions and on two smooth functions}
Since $\theta_1$ and $\theta_2$ are holomorphic forms, they read
\begin{equation}\label{def f1 f2}
\theta_1=f_1dz,\hspace{1cm} \theta_2=f_2dz
\end{equation}
where $f_1$ and $f_2$ are two holomorphic functions; moreover, $f_1$ and $f_2$ do not vanish simultaneously since $dG$ is assumed to be injective at every point. Here and below $z$ is a conformal parameter of $(M,{\mathcal J}),$ as described at the end of Section \ref{ss gauss map}. 
\\

The aim now is to show that the immersion $F:M\stackrel{\varphi}{\hookrightarrow}\R^{1,3}$ is determined by the holomorphic functions $f_1$ and $f_2$, and by the two real functions $h_0,h_1:M\rightarrow\R$ such that $\vec{H}=h_0e_0+h_1e_1,$ the components of the mean curvature vector $\vec{H}$ in the parallel frame $(e_0,e_1).$ We first   observe that the immersion is determined by $g$ and by the orthonormal and parallel frame $(u_1,u_2)$ of $TM:$
\begin{prop}
The immersion $F:M\stackrel{\varphi}{\hookrightarrow}\R^{1,3}$ is such that
\begin{eqnarray*}
dF(X)&=&{g}^{-1}\left(w_1(X)J+w_2(X)K\right)\tch{g}
\end{eqnarray*}
for all $X\in TM,$ where $w_1,w_2:TM\rightarrow\R$ are the dual forms of $u_1,u_2.$
\end{prop}
\begin{proof}
We  have 
$$dF(X)=\langle\langle X\cdot\varphi,\varphi\rangle\rangle={g}^{-1}\ [X]\ \tch{g},$$ 
where $[X]\in\HC$ stands for the coordinates of $X\in Cl(E\oplus TM)$ in $\tilde{s}.$ Recalling that $[u_1]=J$ and $[u_2]=K$ in $\tilde{s},$ we have $[X]=X_1J+X_2K,$ where $(X_1,X_2)$ are the coordinates of $X\in TM$ in $(u_1,u_2);$ the result follows. 
\end{proof}
We finally precise how to recover $g,$ $u_1$ and $u_2$ from $f_1,$ $f_2,$ $h_0$ and $h_1:$
\begin{prop}\label{prop recov g u}
1- $g$ is determined by $f_1$ and $f_2,$ up to the multiplication on the right by a constant belonging to $Spin(1,3).$
\\
\\2- Define $\alpha_1,\alpha_2:M\lgra \C$ such that 
$$u_1=\alpha_1\hspace{1cm}\mbox{and}\hspace{1cm}u_2=\alpha_2$$
in the coordinates $z=x+iy.$ The functions $\alpha_1,\alpha_2,f_1,f_2,h_0$ and $h_1$ are linked by
\begin{equation}\label{relation functions}
(\alpha_1J+\alpha_2K)(f_1J+f_2K)=-ih_0\1+h_1I.
\end{equation}
In particular, if $f_1^2+f_2^2\neq 0,$ we get
\begin{equation}\label{alpha case 1}
\alpha_1J+\alpha_2K=-(-ih_0\1+h_1I)\frac{f_1J+f_2K}{{f_1}^2+{f_2}^2},
\end{equation}
that is, the orthonormal frame $(u_1,u_2)$ in the coordinates $z$ is determined by $f_1,$ $f_2,$ $h_0$ and $h_1.$ 
\end{prop}
\begin{proof}
1- The solution $g$ of the equation $dg\ g^{-1}=\eta'$ is unique, up to multiplication on the right by a constant belonging to $Spin(1,3).$
\\2- Note that, in $\tilde{s},$
\begin{eqnarray*}
[D\varphi]&=&[u_1]\tch{[\nabla_{u_1}\varphi]}+[u_2]\tch{[\nabla_{u_2}\varphi]}\\
&=&\left([u_1]\tch{\eta'}(u_1)+[u_2]\tch{\eta'}(u_2)\right)\tch{g}.
\end{eqnarray*}
since $[\nabla\varphi]=\eta'[\varphi].$ Moreover,
$$[\vec{H}\cdot\varphi]=[\vec H]\tch{g}=(h_0[e_0]+h_1[e_1])\tch{g}.$$
In $\tilde{s},$ the vectors $e_0,$ $e_1,$ $u_1$ and $u_2$ are represented by $i\1,$ $I,$ $J$ and $K$ respectively; the equation $D\varphi=\vec{H}\cdot\varphi$ thus yields
\begin{equation}\label{eqn eta'}
J\tch{\eta'}(u_1)+K\tch{\eta'}(u_2)=ih_0\1+h_1I.
\end{equation}
By (\ref{etap theta}), (\ref{def f1 f2}) and the definition of $\alpha_1$ and $\alpha_2,$ we have
$$\eta'(u_1)=\alpha_1(f_1J+f_2K)\hspace{1cm}\mbox{and}\hspace{1cm}\eta'(u_2)=\alpha_2(f_1J+f_2K),$$
and (\ref{eqn eta'}) implies (\ref{relation functions}). Equation (\ref{alpha case 1}) is a consequence of (\ref{relation functions}), together with the following observation: an element $\xi\in\HC$ is invertible if and only if $H(\xi,\xi)=\overline{\xi}\xi\neq 0;$ its inverse is then $\displaystyle{\xi^{-1}=\frac{\overline{\xi}}{H(\xi,\xi)}.}$
\end{proof}
\begin{rem} The complex numbers $\alpha_1$ and $\alpha_2,$ considered as real vector fields on $M,$  are independent and satisfy $\left[\alpha_1,\alpha_2\right]=0.$ Indeed,  since the metric on $M$ is flat, there exists a local diffeomorphism $\psi:\R^2\rightarrow M$ such that $u_1=\frac{\partial\psi}{\partial x}, u_2=\frac{\partial\psi}{\partial y}.$
\end{rem}
\subsubsection{Interpretation of the condition $f_1^2+f_2^2\neq 0$}
To analyze the meaning of the condition $f_1^2+f_2^2\neq 0$ in the proposition above, we first give the following geometric interpretation of $f_1^2+f_2^2:$
\begin{prop}
Let us define 
$$G^*H:=H(dG,dG),$$ 
where $G:M\rightarrow\mathcal{Q}$ is the Gauss map of $M\subset\R^{1,3}$ with values in 
$$\displaystyle{\mathcal{Q}=\{Z\in\Im m\ \HC:\ H(Z,Z)=1\}.}$$
The complex quadratic form $G^*H$ is the analog of \emph{the third fundamental form} in the classical theory of surfaces in $\R^3.$ We have the formula
\begin{equation}\label{G star H}
G^*H=4(f_1^2+f_2^2)dz^2.
\end{equation}
\end{prop}
\begin{proof}
By Lemma  \ref{lemma expr eta G} and since $G^{-1}=\overline{G}=-G,$ we have $\tilde{\eta}=\frac{1}{2}G^{-1}dG=-\frac{1}{2}GdG$. Thus
$$\tilde{\eta}^2=\frac{1}{4}\ GdG\ GdG.$$
Since $G^2=-1,$ we moreover have $GdG=-dGG.$ Thus
$$\tilde{\eta}^2=\frac{1}{4}\ (-dGG)\ GdG=\frac{1}{4}dG^2=-\frac{1}{4}\overline{dG}dG,$$
that is 
\begin{equation}\label{eta2 1}
\tilde{\eta}^2=-\frac{1}{4}G^*H.
\end{equation}
We also have $\tilde{\eta}^2=g^{-1}\eta'^2g$ (see (\ref{etatilde etap})), with $\eta'^2=(\theta_1J+\theta_2K)^2=-(\theta_1^2+\theta_2^2)\1.$ Thus 
\begin{equation}\label{eta2 2}
\tilde{\eta}^2=-g^{-1}(\theta_1^2+\theta_2^2)\1 g= -(\theta_1^2+\theta_2^2)\1.
\end{equation}
Equations (\ref{eta2 1}) and (\ref{eta2 2}), and the very definition (\ref{def f1 f2}) of $f_1$ and $f_2,$ give the result.
\end{proof}
By (\ref{G star H}), if $f_1^2+f_2^2=0$ at $x_0\in M,$ $dG_{x_0}(T_{x_0}M)$ belongs to
\begin{equation}\label{null lines}
G(x_0)+\{\xi\in\Im m\ \HC:\ H(G(x_0),\xi)=H(\xi,\xi)=0\}.
\end{equation}
This set is the union of two complex lines through $G(x_0)$ in the Grassmannian ${\mathcal Q}$ of the oriented and spacelike planes; these lines have the following geometric meaning: if $N_1,N_2$ are two vectors in $N_{x_0}M$ such that 
$$\langle N_1,N_1\rangle=\langle N_2,N_2\rangle=0\hspace{1cm}\mbox{and}\hspace{1cm}\langle N_1,N_2\rangle\neq 0,$$ 
these lines correspond to the set of spacelike planes belonging to the degenerate hyperplanes $T_{x_0}M\oplus\R N_1$ and $T_{x_0}M\oplus\R N_2$ respectively. More precisely, the following characterization holds: $f_1^2+f_2^2=0$ at $x_0\in M$ if and only if the osculating paraboloid of the surface at $x_0$ belongs to one of the degenerate hyperplanes $T_{x_0}M\oplus\R N_i,$ $i=1,2,$ i.e. is a graph of a quadratic map of the form  
\begin{equation}\label{type q}
(x_1,x_2)\in\R^2\simeq T_{x_0}M\mapsto \left((\zeta-1)x_1^2+(\zeta+1)x_2^2\right)N\ \in\ \R N
\end{equation}
where $N$ is a null vector collinear to $N_1$ or $N_2,$ and where $\zeta$ belongs to $\R;$ in other words, this occurs if and only if the osculating space of the surface at $x_0$ is a degenerate hyperplane. See \cite{BS} for more details.

Since $f_1$ and $f_2$ are holomorphic functions, there is the alternative: 
\begin{center}
$f_1^2+f_2^2$ does not vanish on the complementary set of isolated points of $M,$ 
\end{center}
or
\begin{center}
$f_1^2+f_2^2$ vanishes identically on $M.$ 
\end{center}

In the first  case, (\ref{alpha case 1}) holds, except at some isolated points, and $\alpha_1$ and $\alpha_2$ are thus determined by $f_1,$ $f_2,$ $h_0$ and $h_1$ at every point. It is not difficult to see that the second case is equivalent to the surface being in some degenerate hyperplane $x_0\ +\ T_{x_0}M\oplus\R N$ of $\R^{1,3};$ in the attempt of describing the flat surfaces with flat normal bundle, this case is trivial, since, conversely, every surface belonging to some degenerate hyperplane verifies $K=K_N=0$ (indeed, the values of its Gauss map belong to a complex line in the Grassmannian, precisely, a line of the form (\ref{null lines})).

\subsection{Recovering of a flat immersion with flat normal bundle by Weierstrass data}\label{ss recovering}
We gather the results of the previous section to construct flat immersions with flat normal bundle from initial data, proving Corollary \ref{corollary2 theorem}: we suppose that the hypotheses of Corollary \ref{corollary2 theorem} hold, we consider $E=\R^{1,1}\times U$ the trivial bundle on $U\subset\C,$ where $\R^{1,1}$ is $\R^2$ equipped with the metric $-dx_0^2+dx_1^2,$ and we denote by $(e_0,e_1)$ the canonical basis of $\R^{1,1}.$ Let us define $s=(e_0,e_1,u_1,u_2)$ where $u_1=\alpha_1$ and $u_2=\alpha_2$ in $\R^2\simeq\C$ (the complex numbers $\alpha_1$ and $\alpha_2$ are defined by formulas (\ref{alpha function f h cor}) in Corollary \ref{corollary2 theorem}), and let us consider the metric on $U$ such that $(u_1,u_2)$ is an orthonormal frame; this metric is flat and the frame $(u_1,u_2)$ is parallel since $\left[u_1,u_2\right]=0$ by hypothesis. Let $\tilde{s}$ be a section of the trivial bundle $\tilde{Q}\rightarrow U$ such that $\pi(\tilde{s})=s,$ where $\pi:\tilde{Q}=S^1_{\C}\times U\stackrel{2:1}{\longrightarrow} Q=(SO(1,1)\times SO(2))\times U$ is the natural projection. Equation (\ref{eqn g1}) is solvable since $\eta':=\theta_1J+\theta_2K$ satisfies the structure equation $d\eta'-[\eta',\eta']=0.$ Moreover the solution is unique up to the natural right action of $Spin(1,3)$ (and the solution is unique if some initial value $g(x_0)\in Spin(1,3)$ is given). The definition (\ref{alpha function f h cor}) of $\alpha_1$ and $\alpha_2$ is equivalent to (\ref{alpha case 1}), which traduces that $\varphi:=[\tilde{s},g]\in\Sigma=\tilde{Q}\times\HC/\rho$ is a solution of $D\varphi=\vec{H}\cdot\varphi,$ where $\vec{H}=h_0e_0+h_1e_1$ (see the proof of Proposition \ref{prop recov g u}). Moreover the form $\xi$ defined by (\ref{th def xi}) is such that $\xi(X)=\langle\langle X\cdot\varphi,\varphi\rangle\rangle;$ thus $\xi$ is closed (Proposition \ref{prop fundamental xi}), and a primitive of $\xi$ defines an isometric immersion in $\R^{1,3}\subset\HC$ (Theorem \ref{theorem second integration}). Since the Gauss map of the immersion is $G={g}^{-1}Ig$ (Lemma \ref{etap2}), and since  $g$ is a holomorphic map  (formula (\ref{eqn g1})), we get that $G$ is a holomorphic map too, and thus that $K=K_N=0$ (Proposition \ref{prop pull back}).
\begin{rem}\label{rmq psi}
According to Corollary \ref{corollary2 theorem}, a flat immersion with flat normal bundle and regular Gauss map, and whose osculating spaces are everywhere not degenerate (i.e. such that $G^*H\neq 0$ at every point, see the end of the previous section), is determined by two holomorphic functions $f_1,f_2:U\rightarrow\C$ such that $f_1^2+f_2^2\neq 0$ on $U$ and by two smooth functions $h_0,h_1:U\rightarrow\R$ such that the two complex numbers $\alpha_1$ and $\alpha_2$ defined by (\ref{alpha function f h cor}), considered as real vector fields, are independent at every point and such that $\left[\alpha_1,\alpha_2\right]=0$ on $U.$ Considering further a holomorphic function $h:U\rightarrow\C$ such that $h^2=f_1^2+f_2^2,$ and setting $z'$ for the parameter such that $dz'=h(z)dz,$ we have
$$g^*H=H(dg,dg)=H(dg g^{-1},dgg^{-1})=\left(f_1^2+f_2^2\right)dz^2=d{z'}^2,$$
and thus, in $z',$
\begin{equation}\label{def psi}
g'g^{-1}=\cos\psi\ J+\sin\psi\ K
\end{equation}
for some holomorphic function $\psi:U'\rightarrow\C.$ The parameter $z'$ may be interpreted as the \emph{complex arc length} of the holomorphic curve $g:U\rightarrow Spin(1,3),$ and  the holomorphic function $\psi$ as the \emph{complex angle} of $g'$ in the trivialization $TSpin(1,3)=Spin(1,3)\times\mathcal{G}.$ Observe that, from the definition (\ref{def psi}) of $\psi,$ the derivative $\psi'$ may be interpreted as the \emph{complex geodesic curvature} of the holomorphic curve $g:U\rightarrow Spin(1,3)$. The immersion thus only depends on the single holomorphic function $\psi,$ instead of the two holomorphic functions $f_1,f_2.$ Moreover, the two relations in (\ref{alpha function f h cor}) then simplify to
\begin{equation}\label{alpha function psi h}
\alpha_1=ih_0\cos\psi+h_1\sin\psi\hspace{.5cm}\mbox{and}\hspace{.5cm}\alpha_2=ih_0\sin\psi-h_1\cos\psi.
\end{equation}
Note that the new parameter $z'$ may be only locally defined, since the map $z\rightarrow z'$ may be not one-to-one in general.
\begin{cor}\label{corollary psi}
Let $U\subset\C$ be a simply connected domain, and let $\psi:U\rightarrow\C$ be a holomorphic function. Suppose that $h_0,h_1:U\rightarrow\R$ are smooth functions such that $\alpha_1$ and $\alpha_2,$ real vector fields defined by (\ref{alpha function psi h}), are independent at every point and satisfy $\left[\alpha_1,\alpha_2\right]=0$ on $U.$ Then, if $g:U\rightarrow Spin(1,3)\subset\HC$ is a holomorphic map solving 
$$g'g^{-1}=\cos\psi J+\sin\psi K,$$ 
and if we set
$$\xi:=g^{-1}(\omega_1J+\omega_2K)\hat{g}$$ 
where $\omega_1,\omega_2:TU\rightarrow\R$ are the dual forms of $\alpha_1,\alpha_2\in\Gamma(TU)$ and where $\hat{g}$ stands for the map $g$ composed by the complex conjugation in $\HC,$ the function $F=\int\xi$ defines an immersion $U\rightarrow\R^{1,3}$ with $K=K_N=0.$ Reciprocally, the immersions of $M$ such that $K=K_N=0,$ with regular Gauss map and whose osculating spaces are everywhere not degenerate, are locally of this form.
 \end{cor}
\end{rem}

\subsection{Flat surfaces in three-dimensional hyperbolic space}
We obtain here a spinorial proof of the following result of J. G\'alvez, A. Mart\'{\i}nez and F. Mil\'an concerning the representation of flat surfaces in hyperbolic space $\HH^3:$
\begin{thm}\label{th GMM1}\cite{GMM1,GMM}
Let $B:M\rightarrow SL_2(\C)$ be a map such that 
$$B^{-1}dB=\left(\begin{array}{cc}0&\theta\\\omega&0\end{array}\right),$$
where $\theta$ and $\omega$ are holomorphic 1-forms on $M.$ Assume moreover that $|\theta|\neq|\omega|.$ Then $F=BB^*$ is a flat surface in $\mathbb{H}^3.$ Conversely, every simply-connected flat surface in $\mathbb{H}^3$ may be described in that way. 
\end{thm}
In the theorem and below $B^*$ stands for the matrix $\overline{B}^t;$ moreover, $\mathbb{H}^3\subset\R^{1,3}$ is described as
$$\mathbb{H}^3=\{\Phi\Phi^*:\ \Phi\in SL_2(\C)\}\subset Herm(2),$$
where the Minkowski norm on the space of the $2\times 2$ hermitian matrices $Herm(2)$ is $-\det.$\begin{proof}
The fact that $F=BB^*$ defines a flat surface in $\mathbb{H}^3$ may be proved by a direct computation; see \cite{GMM1}, formula (18). We thus only prove the converse statement. Let $F:M\rightarrow\HH^3\subset\R^{1,3}$ be a flat immersion of a simply-connected surface, $E$ its normal bundle, $\vec H\in\Gamma(E)$ its mean curvature vector field and $\Sigma:=M\times\HC$ the bundle of spinors of $\R^{1,3}$ restricted to $M.$ The immersion $F$ is given by 
$$F=\int\xi,\hspace{1cm}\mbox{where}\hspace{.5cm}\xi(X)=\langle\langle X\cdot\varphi,\varphi\rangle\rangle,$$
for some spinor field $\varphi\in\Gamma(\Sigma)$ solution of $D\varphi=\vec H\cdot\varphi$ and such that $H(\varphi,\varphi)=1$ (the spinor field $\varphi$ is the restriction to $M$ of the constant spinor field $\textit{1}$ or $-\textit{1}\in\HC$ of $\R^{1,3}$, see Remark \ref{rem immersed surface}). Recalling Proposition \ref{prop cond hypersurfaces}, 3-, we have
\begin{equation}\label{F phi H3 fin}
F=\langle\langle e_0\cdot\varphi,\varphi\rangle\rangle
\end{equation}
where $e_0\in\Gamma(E)$ is the future-directed vector which is normal to $\HH^3$ in $\R^{1,3}$. We choose a parallel frame $\tilde{s}\in\Gamma(\tilde{Q})$ adapted to $e_0,$ i.e. such that $e_0$ is the first vector of $\pi(\tilde{s})\in\Gamma(Q_1\times_M Q_2):$ in $\tilde{s},$ equation (\ref{F phi H3 fin}) reads
\begin{equation}\label{F phi H3 fin bis}
F=i\ \overline{[\varphi]}\ \tch{[\varphi]},
\end{equation}
where $[\varphi]\in\HC$ represents $\varphi$ in $\tilde{s}.$
We now consider the isomorphism of algebras
\begin{eqnarray*}
A:\hspace{1cm}\HC&\rightarrow& M_2(\C)\\
q=q_0\textit{1}+q_1I+q_2J+q_3K&\mapsto&A(q)=\left(\begin{array}{cc}q_0+iq_1&q_2+iq_3\\-q_2+iq_3&q_0-iq_1\end{array}\right).
\end{eqnarray*}
We note the following properties: 
$$A(\tch{\overline{q}})=A(q)^*\hspace{1cm}\mbox{and}\hspace{1cm}H(q,q)=\det(A(q))$$
for all $q\in\HC,$ and
$$A(\R^{1,3})=iHerm(2)$$
where $\R^{1,3}=\{\xi\in\HC:\ \tch{\overline{\xi}}=-\xi\}.$ Thus, setting  $B:=A(\overline{[\varphi]}),$ $B$ belongs to  $SL_2(\C)$ (since $H(\varphi,\varphi)=1$) and $B^*=A(\tch{[\varphi]}).$ From (\ref{F phi H3 fin bis}) we thus get $F\simeq iBB^*.$ Dropping the coefficient $i,$ we get that $F$ identifies to $BB^*\in Herm(2).$ Note that $dB=A(d\overline{[\varphi]}),$ and thus that 
\begin{eqnarray*}
B^{-1}dB&=&A([\varphi]d\overline{[\varphi]})=-A(d[\varphi]\ \overline{[\varphi]})\\
&=&-A(\theta_1J+\theta_2K)=-\left(\begin{array}{cc}0&\theta_1+i\theta_2\\-\theta_1+i\theta_2&0\end{array}\right),
\end{eqnarray*}
where $\theta_1$ and $\theta_2$ are holomorphic 1-forms; see Theorem \ref{prop etap}. Finally, it is not difficult to verify that $dF$ injective reads $|\theta_1+i\theta_2|\neq |-\theta_1+i\theta_2|.$
\end{proof}
\begin{rem}
From the proof, it appears that the matrix $B$ in Theorem \ref{th GMM1} has the following interpretation: $B^{-1}$ represents the constant spinor field $\textit{1}$ or $-\textit{1}\in\HC$ of the ambient space $\R^{1,3},$ restricted to the surface, in a parallel frame adapted to the surface and to the embedding $\mathbb{H}^3\hookrightarrow\R^{1,3}.$ 
\end{rem}
\textbf{Acknowledgements}
The author is very indebted to Marie-Am\'elie Lawn and to Julien Roth for many discussions about spinors in Euclidean 4-space: the present paper would not have been possible without the preparation of the paper \cite{BLR}.

\end{document}